\documentclass[11pt]{article}
\usepackage{graphicx}
\usepackage[utf8]{inputenc}
\usepackage{amsmath,amsfonts,amssymb,amscd,amsthm,txfonts,hyperref,yfonts}

\newtheorem{theorem}{Theorem}[section]
\newtheorem{proposition}[theorem]{Proposition}
\newtheorem{lemma}[theorem]{Lemma}

\newtheorem{remark}{Remark}[section]
\newtheorem{theo}{Theorem}

\theoremstyle{definition}
\newtheorem{definition}[theorem]{Definition}

\newcommand{\R}{\mathbb{R}}
\newcommand{\C}{\mathbb{C}}
\newcommand{\N}{\mathbb{N}}
\newcommand{\Z}{\mathbb{Z}}
\newcommand{\T}{\mathbb{T}}

\newcommand{\lcurl}{\lesssim}

\begin{document}

\title
{On the use of normal forms in the propagation of random waves}
\author{Anne-Sophie de Suzzoni\footnote{Universit\'e de Cergy-Pontoise,  Cergy-Pontoise, F-95000,UMR 8088 du CNRS }}

\maketitle

\begin{abstract} We consider the evolution of the correlations between the Fourier coeficients of a solution of the Kamdostev-Petviavshvili II equation when these coefficients are initially independent random variables. We use the structure of normal forms of the equation to prove that those correlations remain small until times of order $\varepsilon^{-5/3}$ or $\varepsilon^{-2}$ depending on the quantity considered.\end{abstract} 

\tableofcontents

\section{Introduction}

We consider the Kadomtsev-Petviashvili II equation with a small non linearity on the torus of dimension 2, $\T^2$, that is
\begin{equation}\label{introkp}
\partial_x \Big( \partial_t u + \partial_x^3 u + \varepsilon u\partial_x u \Big) + \partial_y^2 u = 0
\end{equation}
with $\varepsilon \ll 1$, when the initial datum $u_0 := u(t=0)$ is a random variable with values in $H^s$.

The main aim of this paper is to extend the result of \cite{dSTont} to longer times and lower regularity on KP-II using a normal form version of this equation. But we require more integrability of the initial datum in the probability space. Normal forms have often been used to obtain better time scales in function of the size of the non linearity of a PDE (or equivalently, the size of the initial datum), see, for instance, \cite{Dexi,Snormalforms}.

The KP-II equation models shallow water waves in the approximation of long wave length in the $x$ direction, and when the surface tension is weak. It is known to be well-posed in $H^s$, $s\geq 0$, see \cite{Bont}.

Motivated by the notion of statistical equilibrium in the Physics literature, see \cite{FZwea,ZDPone}, and also by \cite{MMcLTaon}, we assume that the Fourier coefficients of the initial datum $u_0$ are independent from each other and that their laws are invariant by multiplication by $e^{i\theta}$, for all $\theta$. We recall that the statistical equilibrium of a system modelled by the solution $u(t)$ of a Hamiltonian equation is reached when the expectations $E(|u_n(t)|^2)$ of the amplitudes to the square of the Fourier coefficients of $u(t)$ do not depend on time. This notion is introduced in \cite{FZwea}.We are interested in the persistence of the decorrelation between the Fourier coefficients of the solution. Indeed, to compute the random variable corresponding to statistical equilibrium, the expectations of products of the Fourier coefficients are approached by the products of the expectations. We want to know in which sense this approximation is true. A reasonable choice of quantities to study this decorrelation would then be the moments of the solution $u$, that is, writing 
$$
u = \sum_{n\in \Z^2} u_n(t) e^{i(n_xx + n_y y)}\; ,
$$
these are defined for all $p \in \N$ and all $(n_1,\hdots,n_p) \in \Z^p$ as 
$$
E\Big( \prod_{i=1}^p u_{n_i}(t)\Big)
$$
where $E$ is the expectation. However, the bigger $p$ is, the more complex the computation seems to be, thus we focus on $p=2$ and $p=3$. In \cite{dSTont}, the authors considered $p=2$, but we think that adding the moments of order 3 (i.e. when $p=3$) gives some insight regarding what happens when $p$ is not fixed, as the behaviour of the moments of order $p$ is partially dictated by whether $p$ is odd or even. We comment this in the last subsection of this paper.

What we do is that we expand
$$
E\Big( u_n \overline{u_m}\Big) 
$$
at order $3$ in $\varepsilon$ and  
$$
E\Big( u_nu_mu_p\Big)
$$
at order 2 in $\varepsilon$. We keep track of the dependence in time.

We denote by $i\omega_k$ the eigenvalue of $-\partial_x^3 - \partial_x^{-1}\partial_y^2$ associated to the wavelength $k=(k_1,k_2)$, that is 
$$
\omega_k = k_1^3 - \frac{k_2^2}{k_1}
$$
and by $\Delta_n^{k,l}$ the difference $\omega_k+\omega_l - \omega_n$ when $k+l=n$. This quantity $\Delta_n^{k,l}$ is the frequency of the three waves interaction $k,l \rightarrow k+l$.

\begin{theo}\label{th-mainmomtwo} Fix $s>1$. Let $(g_n)_{n\in \N^*\times \Z}$ a sequence of independent identically distributed random variables whose law is invariant by multiplication by $e^{i\theta}$. Let $(\lambda_n)_{n\in \N^*\times \Z}$ be a sequence of complex numbers. Set
$$
\lambda_{-n} = \overline{\lambda_n} \mbox{ and } g_{-n} = \overline{g_n}
$$
such that 
$$
u_0 = \sum_{n\in \Z^* \times \Z} g_n \lambda_n e^{i(n_x x+ n_y y)}
$$
is real-valued. Assume that $u_0$ belongs to $L^\infty (\Omega, H^s(\T^2))$ and that its norm is fixed at 1. Call
$$
u(t) = \sum_n u_n(t) e^{i(n_x x+ n_y y)}
$$
the solution of KP-II with initial datum $u_0$. 

The expansion in $\varepsilon$ of $E( u_n \overline{u_m})$ is given by
$$
E( u_n \overline{u_m}) = \delta_n^m |\lambda_n|^2 + \varepsilon^2 F_{n,m}(t) + \varepsilon^4 \widetilde R(n,m,t,\varepsilon)
$$
with 
\begin{eqnarray*}
F_{n,n}(t) &= &  - n_1 E(|g_n|^2)^2 \sum_{k +l = n} \frac{\cos(\Delta_n^{k,l}t)-1}{(\Delta_n^{k,l})^2} \Big( k_1 |\lambda_n|^2|\lambda_l|^2 + l_1 |\lambda_n|^2|\lambda_k|^2 - n_1 |\lambda_k|^2|\lambda_l|^2\Big)  \\
  & & -n_1 (E(|g_n|^4)-2 E(|g_n|^2)^2) \Big( 2n_1 \frac{\cos(\Delta_n^{-n,2n}t)-1}{(\Delta_n^{-n,2n})^2} |\lambda_n|^4 - \\
 & & \delta_{n/2 \in \N^2}\frac{n_1}{2} \frac{\cos(\Delta_n^{n/2,n/2}t)-1}{(\Delta_n^{n/2,n/2})^2} |\lambda_{n/2}|^4\Big) \; .
\end{eqnarray*} 
and $F_{n,m}(t) = 0$ if $n\neq m$.  Moreover,
$$
\sum_{n,m} \sqrt{|n_1m_1|} |nm|^{s}|F_{n,m}(t)| 
$$
is bounded uniformly in time and there exists $T_1 > 0$ and $\varepsilon_1 >0$, such that for all $\varepsilon \in [0,\varepsilon_1[$ and all $t\in [-T,T]$, with $T = T_1 \varepsilon^{-5/3}$,
$$
|\widetilde R(\varepsilon,n,m,t)| \leq C_s (\min (|n|,|m|))^{-s}|t| (1+|t|)^{7/5}\; .
$$
\end{theo}

Notice that when $n\neq m$
$$
E\Big( u_n(t) \overline{u_m(t)} \Big)
$$
is null up to third order, which gives some credit to the persistence of the decorrelation. 

The term $\varepsilon^4 \widetilde R$ is small as long as the time is a $o(\varepsilon^{-5/3})$, and this is how we get longer times than \cite{dSTont}, in which the time scale is $o(\varepsilon^{-1}$). But we can notice that the bound on the term of second order in $\varepsilon$, $\varepsilon^2 F_{n,m}(t)$ is uniform in time, hence this term is bounded by $\varepsilon^2$. For the remainder to be smaller than this term, and make the expansion in $\varepsilon$ an actual expansion for the times we consider, these times have to be at most $o(\varepsilon^{-5/6})$. We can also remark that, according to \cite{dSTont}, it appears that in the case of an equation presenting resonances within the three waves interaction, that is, when $\Delta_n^{k,l}$ can be $0$, KP-I for instance, the expansion leads to terms of the type $\varepsilon^n|t|^n$ as long as $t$ is of order less than $\varepsilon^{-1}$, which makes it an actual expansion for the natural time scale. This gives estimates explaining why the non resonant terms are said to be negligible in the Physics literature.

Besides, as we use a contraction argument we get the same regularity on the solution $u$ and the initial datum $u_0$ which enables us to assume that $s>1$ instead of $s>2$. However, we have to assume, in order to perform the contraction argument, that the initial datum is more integrable than in \cite{dSTont}, but it does not seem to contradict the assumptions made in the Physics literature.

In order to prove the theorem, we expand $u$ into its three first Picard interactions
$$
u= a + \varepsilon b + \varepsilon^2 c + \varepsilon^3 d(\varepsilon)
$$
where $a$ is the solution of the linear equation
$$
\partial_x (\partial_t a + \partial_x^3 a ) + \partial_y a  = 0
$$
with initial datum $u_0$, $b$ is the solution of 
$$
\partial_x (\partial_t b + \partial_x^3 b + a\partial_x a) + \partial_y b  = 0 
$$
with initial datum $0$, $c$ is the solution of
$$
\partial_x ( \partial_t c + \partial_x^3 c + \partial_x (ab)) + \partial_y^2 c =0
$$
with initial datum $0$ and $d$ is what is left.

The structure of $F_{n,m}(t)$ is derived from computations using the formulae giving $a$ and $b$. The estimates on the remainder $\widetilde R(n,m,t,\varepsilon)$ require to use normal forms. As in \cite{Tlon}, we transform the equation on $u$ into an equation on $v = u+\varepsilon S(u,u)$ of the form
$$
v_t + L v = \varepsilon^2 F(u,u,u)
$$
where $L$ is a linear map and $F$ and $S$ are multi linear maps. We estimate $d$ by bounding the term of order 3 in $v$.

The gain on time comes from bounds on the source term of the equation using normal forms that are better (of size $\varepsilon (1+|t|)$) than the one involved in the equation solved by $d$, which is of size $1+|t|$.

Nevertheless, the moments of order 3 contradicts the persistence of the decorrelation as it should be null if the considered Fourier coefficients $u_n$, $u_m$ and $u_p$ are independent. But, regarding their expansion, we get the following result.

\begin{theo}\label{th-mainmomthree}
The expansion in $\varepsilon$ of $E( u_nu_mu_p)$ is given by
$$
E( u_nu_mu_p)  = \varepsilon F_{n,m,p}(t) + \varepsilon^3 \widetilde R(n,m,p,t,\varepsilon)
$$
with
\begin{eqnarray*}
F_{n,m,p}(t) & = & -\frac{1-e^{i(\omega_n + \omega_m + \omega_p)t}}{\omega_n + \omega_m + \omega_p} \left( E(|g_n|^2) \Big( n_1 |\lambda_m|^2|\lambda_p|^2 + m_1 |\lambda_p|^2|\lambda_n|^2 + p_1 |\lambda_n|^2|\lambda_m|^2 \Big)+ \right. \\
 & & \left. (E(|g_n|^4 - 2 E(|g_n|^2)^2) \Big( \delta_m^p m_1 |\lambda_m|^4 + \delta_p^n p_1 |\lambda_p|^4 + \delta_n^m n_1 |\lambda_n|^4 \Big) \right)\; .
\end{eqnarray*}
when $n+m+p=0$ and $F_{n,m,p}(t) = 0$ otherwise. Moreover,
$$
\sum_{n,m,p} \sqrt{|n_1m_1p_1|}|nmp|^{s}|F_{n,m,p}(t)| 
$$
is bounded uniformly in time and there exist $\varepsilon_1>0$ and $T_1>0$ such that for all $\varepsilon < \varepsilon_1$ and $t \in [-T,T]$, with $T=T_1\varepsilon^{-2}$,
$$
|\widetilde R(\varepsilon, n,m,p,t) |\leq C_s (\min(|nm|,|mp|,|pn|))^{-s} |t|(1+|t|) \; .
$$\end{theo}

We do not treat $p=1$ but the same computation as we will perform leads to the fact that the mean value of $u_n(t)$ is null up to order 2 in $\varepsilon$ and the remainder term is bounded by $C_s(1+|t|)|t|$. 

What is more, if the norm of $u_0$ is equal to $\mu$ instead of $1$, by replacing $u$ by $v= \frac{u}{\mu}$ and applying the theorem on $v$, we get that the estimates are valid until time $T = T_1  (\varepsilon \mu)^{-5/3}$ or $T = T_1  (\varepsilon \mu)^{-2}$ and $\widetilde R(n,m,t,\varepsilon)$ is bounded by $C_s \mu^2 |t| (1+|t|)$, and $\widetilde R(n,m,p,t,\varepsilon)$ is bounded by $C_s \mu^3 |t|(1+|t|)$.

\paragraph{Plan of the paper} The paper is organized as follows.

In Section 2, we define the problem more precisely and we compute the derivatives in time of $F_{n,m}(t)$ and $F_{n,m,p}(t)$. 

In Section 3, we reduce KP-II thanks to the normal form technique and get a bound for $d$. The main difference with \cite{Tlon} is that we need estimates independent from $\varepsilon$, and that we consider the expansion of the reduction.

In Section 4, we prove the estimates on $\widetilde R(n,m,t,\varepsilon)$,  $\widetilde R(n,m,p,t,\varepsilon)$, $F_{n,m}(t)$, and $F_{n,m,p}(t)$. At the end of this section, we propose to compare the result to formally invariant measures, using particular values for $\lambda_n$ and $g_n$. We then mention a possible form of the expansion of the moments of higher order.

\section{Expansion of the solution and formal computations}\label{sec-devsol}

In this section, we start by defining the objects we compute, and stating the assumptions on the initial datum. Then, we do the formal computations using the assumptions of independence and invariance by rotation of the initial datum.

\subsection{Definition of the problem and probabilistic assumptions on the initial datum}\label{sub-defpro}

We consider the Cauchy problem associated to KP-II with a weak non linearity on the torus of dimension 2, $\T^2$, that is : 
\begin{equation}\label{eq-kpii}\left \lbrace{\begin{tabular}{ll}
$\partial_x \left( \partial_t u + \partial_x^3 u + \varepsilon \frac{1}{2}\partial_x(u^2) \right) + \partial_y^2 u = 0$ \\
$u(t=0) = u_0$ \end{tabular}} \right. \; .\end{equation}
We suppose that $\varepsilon \ll 1$ and that the initial datum is a random variable on a probability space $(\Omega,\mathcal A, \mathbb P)$.

We assume that the mean value of the solution $u$ along its first variable $x$ is $0$, as it is a property preserved by the flow of KP-II. In other terms, we assume that the Fourier coefficients $u_{(0,n_2)}$ of the solution are zero and work on Sobolev spaces of functions satisfying this property. 

\begin{definition} We call $H^s$ the topological space of functions $u$ such that $\int_{\T}udx = 0$, that $u$ is real valued, and induced by the norm : 
$$
\|\sum_n u_n e^{inz}\|_{H^s} = \sqrt{\sum_{n_1\neq 0} |n|^{2s} |u_n|^2}
$$
with $n=(n_1,n_2)$, $z=(x,y)$, $nz=n_1x+n_2y$, $|n| = |n_1|+|n_2|$ and where $u_n$ is the Fourier coefficient of $u$ associated to the space frequency $n$.\end{definition}

In this space, we can write \eqref{eq-kpii} considering where the initial datum lives as 
\begin{equation}\label{eqde-kpii}\left \lbrace{\begin{tabular}{ll}
$\partial_t u + \partial_x^3 u + \varepsilon \frac{1}{2}\partial_x(u^2) + \partial_x^{-1} \partial_y^2 u = 0$ \\
$u(t=0) = u_0\in L^\infty(\Omega,H^s)$ \end{tabular}} \right. \; .\end{equation}

Writing the solution $u (t,x,y) = \sum_n u_n(t) e^{inz}$, we aim to develop the mean values
\begin{equation}\label{aim}
E(u_n(t)\overline{u_m(t)})\: , \; E(u_n(t)u_m(t)u_p(t))
\end{equation}
where $E$ is the expectation with regard to the probability space $(\Omega,\mathcal A, \mathbb P)$ in their different orders in $\varepsilon$ up to order 3 for the former one and 2 for the latter.

We make some assumptions on the initial datum $u_0$. We assume that $u_0$ can be written : 
$$
u_0(x,y) = \sum_{n_1\neq 0} \lambda_n g_n e^{inz}
$$
where $(\lambda_n)$ is a sequence of complex numbers and $(g_n)$ a sequence of random variables from $\Omega$ to $\C$. We assume that $u_0$ belongs to $L^\infty(\Omega, H^s)$, with $s>1$. To remain in a real framework, we impose that
$$
\lambda_{-n} = \overline{\lambda_n}\; , \; g_{-n} = \overline{g_n}\; .
$$
where $-n = (-n_1,-n_2)$. Finally, we assume that the $g_n$ for $n_1> 0$ are all independent from each other, have the same law and that this law is invariant by all the rotations, that is for all $\theta \in [0,2\pi]$, $e^{i\theta}g_n$ has the same law as $g_n$.

\begin{remark}It is common (when we do not have the $L^\infty$ assumption) to consider that the $g_n$ should be complex centred and normalized Gaussian variables. In this case, $g_n$ can be written as $g_n = h_n+il_n$ where $h_n$ and $l_n$ are real centred Gaussian variable independent from each other. However, in the general case, if $g_n$ is separated in its real and imaginary parts as $h_n+il_n$, to assume that $h_n$ and $l_n$ have the same law, are independent and invariant by multiplication by $-1$ will not guarantee that $g_n$ is invariant by rotation. Indeed, in this case the mean value of $g_n^4$ is equal to $2E(h_n^4)-6E(h_n^2)^2$ instead of $0$ if $g_n$ is invariant by rotation. The invariance by rotation is a crucial ingredient in the formal computation done in the next subsections. This should explain why we kept the complex structure of the solution instead of writing it in the basis obtained with the sines and cosines.\end{remark}

We now explain how we intend to make the afore-mentioned expansions \eqref{aim}.

We expand $u$ into its first Picard interactions, which means that we write
$$
u = a + \varepsilon b + \varepsilon^2 c + \varepsilon^3d(\varepsilon)
$$
where $a$ is the solution of the linearised around $0$ equation of KP-II
$$
\partial_x \left( \partial_t a + \partial_x^3 a \right) + \partial_y^2 a = 0
$$
with initial datum $u_0$, where $b$ is the solution of 
$$
\partial_x \left( \partial_t b + \partial_x^3 b + \frac{1}{2}\partial_x(a^2) \right) + \partial_y^2 b = 0
$$
with initial datum $0$, $c$ is the solution of
$$
\partial_x \Big( \partial_t c + \partial_x^3c + \partial_x (ab) \Big) + \partial_y^2 c = 0
$$
with initial datum $0$ and $d$ is what is left, which means that $d$ is the solution of 
\begin{equation}\label{eqond}
\partial_x \left( \partial_t d + \partial_x^3 d + \frac{1}{2}\partial_x(b^2 + 2ac + 2\varepsilon (bc + ad) +\varepsilon^2 (c^2+2bd) + 2\varepsilon^3 cd + \varepsilon^4 d^2) \right)
 + \partial_y^2 d = 0
\end{equation}
with initial datum $0$, hence depending on $\varepsilon$, unlike $a$, $b$ and $c$. 

We have explicit expressions of $a$, $b$ and $c$ depending on $u_0$, which will enable us to do the computations of the first orders  in $ \varepsilon$ of $E(u_n\overline{u_m})$ and $E(u_nu_mu_p)$.

Let us give further notations.

We write $\omega_n = n_1^3 - \frac{n_2^2}{n_1}$ such that $i\omega_n$ is the eigenvalue of
$$
L = -\partial_x^3 - \partial_x^{-1}\partial_y^2 \
$$
associated to $e^{inz}$.

The flow of the equation $\partial_t u = Lu$ is then denoted by $U(t)$ and we have
$$
U(t) \left( \sum_n u_n e^{inz}\right) = \sum_n e^{i\omega_n t } u_n e^{inz} \; .
$$
Hence, $a$ is equal to $U(t)u_0$ and its Fourier coefficient $a_n$ is given by $e^{i\omega_n t}\lambda_n g_n$.

It is known that KP-II present no resonances regarding three waves interaction in the sense of the following proposition.

\begin{proposition} For all $k,l,n\in \Z^2$ such that $k_1$, $l_1$ and $n_1$ are different from zero and $k+l=n$, we have that
$$
|\omega_n -\omega_k -\omega_l |\geq 3|n_1k_1l_1| \; .
$$
\end{proposition}

\begin{proof}The proof is a straightforward computation. \end{proof}

We then write $\Delta_n^{k,l} = \omega_k +\omega_l -\omega_n$ and use this notation to describe $b$. 

\begin{lemma}\label{lem-formob} The order 1 in $\varepsilon$ of the solution $u$ of KP-II is 
$$
b(t) = \sum_n b_n(t) e^{inz}
$$
where
$$
b_n(t) = -\frac{n_1}{2} \sum_{k+l=n} e^{i\omega_n t }\frac{e^{i\Delta_n^{k,l}t}-1}{\Delta_n^{k,l}} \lambda_k\lambda_l g_k g_l\; .
$$
\end{lemma}

\begin{proof} We write the equation satisfied by $b_n(t)$ : 
$$
\dot b_n(t) = i \omega_n b_n(t) - \frac{in_1}{2} \sum_{k+l=n} a_k a_l\; ,
$$
and then replace $a_ka_l$ by its value to get
$$
\dot b_n(t) = i \omega_n b_n(t) - e^{i\omega_n t} \frac{in_1}{2} \sum_{k+l=n}e^{i\Delta_n^{k,l}t}\lambda_k\lambda_l g_k g_l
$$
with initial datum $0$, which we integrate to get the result, as 
$$
b_n(t) e^{-i\omega_n t} = -\frac{in_1}{2}\int_{0}^t  \sum_{k+l = n} e^{i\Delta_n^{k,l}t'} \lambda_k\lambda_l g_k g_l \; .
$$\end{proof}

\begin{lemma} The order 2 in $\varepsilon$ of the solution $u$ of KP-II is
$$
c = \sum_n c_n(t) e^{inz}
$$
where 
$$
c_n(t) = in_1 e^{i\omega_n t }\sum_{j+k+l = n} \Big[ \int_{0}^t \left( e^{i(\omega_j+\omega_k+\omega_l-\omega_n)t'} - e^{i\Delta_n^{j,k+l}t'}\right) dt' \Big] \frac{k_1+l_1}{2 \Delta_{k+l}^{k,l}} \lambda_j \lambda_k \lambda_l g_j g_kg_l \; .
$$
\end{lemma}

\begin{proof} Since $c$ solves 
$$
\partial_t c -L c + \partial_x (ab) = 0
$$
with initial value $0$, it can be written 
$$
c(t) = - \int_{0}^t U(t-t') \partial_x (ab) dt' \; .
$$
Hence, its $n$-th Fourier coefficients is given by
$$
c_n(t) = - \int_0^t e^{i\omega_n (t-t')} in_1 \sum_{k+l=0} a_k(t') b_l(t')dt'
$$
and by replacing $a_k$ and $b_l$ by their values, we get
$$
c_n(t)= in_1 e^{i\omega_n t } \sum_{k+l = n} \int_{0}^t \left( e^{i\Delta_n^{k,l}t'} \lambda_k g_k \frac{l_1}{2}\sum_{j+q = l}\frac{e^{i\Delta_l^{j,q}t'}-1}{\Delta_l^{j,q}} \lambda_j\lambda_q g_j g_q\right) dt'
$$
and, suppressing the use of the notation $l$, as $\Delta_n^{k,l} + \Delta_l^{j,q}  = -\omega_n + \omega_j + \omega_k + \omega_q$ and $ \Delta_n^{k,l} = \Delta_n^{k,j+q}$, we get
$$
c_n(t) = in_1 e^{i\omega_n t } \sum_{j+q+k = n} \left(\int_{0}^t \Big[ e^{i(-\omega_n + \omega_j + \omega_k + \omega_q) t'} - e^{i\Delta_n^{k,j+q}t'} \Big]dt'\right) \frac{j_1+q_1}{2\Delta_{j+q}^{j,q}} \lambda_j\lambda_k\lambda_q g_jg_kg_q\; .
$$
\end{proof}

\subsection{Expansion of the moments of order 2}\label{sub-momtwo}

The following lemma sums up the formal computation related to $E(u_n\overline{u_m})$. Let us remark that they are essentially the same as in \cite{dSTont}.

\begin{lemma}The expansion of $\partial_t E(e^{i(\omega_m -\omega_n)t}u_n\overline{u_m})$ is written : 
$$
\partial_t E(e^{i(\omega_m -\omega_n)t}u_n\overline{u_m}) =\varepsilon^2  \delta_n^m  G_n(t) + \varepsilon^4 R(\varepsilon, n,m,t)
$$
with the term of order 2 in $\varepsilon$ given by
\begin{eqnarray*}
G_n(t) & = &  - n_1 E(|g_n|^2)^2 \sum_{k +l = n} \frac{\sin(\Delta_n^{k,l}t)}{\Delta_n^{k,l}} \Big( k_1 |\lambda_n|^2|\lambda_l|^2 + l_1 |\lambda_n|^2|\lambda_k|^2 - n_1 |\lambda_k|^2|\lambda_l|^2\Big) + \\
 & & -n_1 (E(|g_n|^4)-2 E(|g_n|^2)^2) \Big( 2n_1 \frac{(\sin(\Delta_n^{-n,2n}t)}{\Delta_n^{-n,2n}} |\lambda_n|^4 - \delta_{n/2 \in \N^2}\frac{n_1}{2} \frac{\sin(\Delta_n^{n/2,n/2}t)}{\Delta_n^{n/2,n/2}} |\lambda_{n/2}|^4\Big) \; .
\end{eqnarray*}
where $\delta_n^m$ is the Kronecker symbol, equal to $1$ if $n=m$, $0$ otherwise.

What is more, the remainder satisfies, if $s\geq 1$
\begin{eqnarray*}
|R(\varepsilon,n,m,t)| & \leq & C \min(|n|,|m|)^{-s} \Big( \|a\|^2_{L^\infty,H^s}\|d\|_{L^\infty,H^s} + \|a\|_{L^\infty,H^s}\|b\|_{L^\infty,H^s}\|c\|_{L^\infty,H^s}+\|b\|_{L^\infty,H^s}^3\\
 & & \varepsilon (\|a\|_{L^\infty,H^s}\|b\|_{L^\infty,H^s}\|d\|_{L^\infty,H^s}+\|a\|_{L^\infty,H^s}\|c\|_{L^\infty,H^s}^2+\|b\|_{L^\infty,H^s}^2\|c\|_{L^\infty,H^s} ) + \\
 & & \varepsilon^2 (\|a\|_{L^\infty,H^s}\|c\|_{L^\infty,H^s}\|d\|_{L^\infty,H^s}+\|b\|_{L^\infty,H^s}^2\|d\|_{L^\infty,H^s}+\|b\|_{L^\infty,H^s}\|c\|_{L^\infty,H^s}^2) + \\
 & & \varepsilon^3 (\|b\|_{L^\infty,H^s}\|c\|_{L^\infty,H^s}\|d\|_{L^\infty,H^s}+\|c\|_{L^\infty,H^s}^3) + \\
 & & \varepsilon^4 (\|b\|_{L^\infty,H^s}\|d\|_{L^\infty,H^s}^2+\|c\|_{L^\infty,H^s}^2\|d\|_{L^\infty,H^s}) + \\
 & & \varepsilon^5 \|c\|_{L^\infty,H^s}\|d\|_{L^\infty,H^s}^2 + \\
 & & \varepsilon^6 \|d\|_{L^\infty,H^s}^3 \Big)\; ,
\end{eqnarray*}
where the $L^\infty$ norm is taken on the probability space $\Omega,\mathcal A, \mathbb P$ and the $H^s$ norm on the torus $\T^2$.
\end{lemma}

\begin{remark}It is noticeable that the term of order 1 is always null and that the term of order 2 is null when $n\neq m$. \end{remark}

\begin{proof} We begin by writing $\partial_t E(e^{i(\omega_m -\omega_n)t}u_n\overline{u_m})$ as :
$$
\partial_t E(u_n\overline{u_m}) = A_{n,m}(t) + \overline A_{m,n} (t)
$$
with
$$
A_{n,m}(t) =   \varepsilon \frac{im_1}{2} \sum_{k+l = m} e^{i(\omega_m -\omega_n)t}E(u_n \overline{u_k u_l}) \; ,
$$
by replacing $\overline{\partial_t e^{-i\omega_m t} u_m}$ by its expression as $u$ is a solution of KP-II, that is
$$
\overline{\partial_t e^{-i\omega_m t} u_m} = \overline{-e^{-i\omega_m t } \frac{im_1}{2}\sum_{k+l = m}u_ku_l}\; .
$$

We then expand $A_{n,m}(t)$ in terms of $\varepsilon$ as
$$
A_{n,m}(t)=  \sum_{j=1}^{10} \varepsilon^j A_{n,m}^j (t)
$$
by expanding $u_n$ in $a_n+\varepsilon b_n + \varepsilon^2 c_n + \varepsilon^3d_n(\varepsilon)$ and where $A_{n,m}^j(t)$ is of the form 
$$
A_{n,m}^j(t) = \frac{im_1}{2}e^{i(\omega_m-\omega_n)t}\sum_{k+l=m} E(\alpha_n\overline{\beta_k \gamma_l})
$$
and where $\alpha$, $\beta$, and $\gamma$ are replaced with either $a$, $b$, $c$, or $d$. The index $j$ is equal to 1 plus the number of occurrences of $b$ plus twice the one of $c$, plus thrice the one of $d$. This is summed up in the next table.

\bigskip

\begin{tabular}{|c|c|c|c|c|}
Occurrences of $a$ & Occurrences of $b$ & Occurrences of $c$ & Occurrences of $d$ & Order in $\varepsilon$ \\ \hline
$0$ & $0$ & $0$ & $3$ & $10$ \\ \hline 
$0$ & $0$ & $1$ & $2$ & $9$ \\ \hline
$0$ & $0$ & $2$ & $1$ & $8$ \\ \hline 
$0$ & $0$ & $3$ & $0$ & $7$ \\ \hline
$0$ & $1$ & $0$ & $2$ & $8$ \\ \hline
$0$ & $1$ & $1$ & $1$ & $7$ \\ \hline
$0$ & $1$ & $2$ & $0$ & $6$ \\ \hline
$0$ & $2$ & $0$ & $1$ & $6$ \\ \hline 
$0$ & $2$ & $1$ & $0$ & $5$ \\ \hline
$0$ & $3$ & $0$ & $0$ & $4$ \\ \hline
$1$ & $0$ & $0$ & $2$ & $7$ \\ \hline
$1$ & $0$ & $1$ & $1$ & $6$ \\ \hline
$1$ & $0$ & $2$ & $0$ & $5$ \\ \hline
$1$ & $1$ & $0$ & $1$ & $5$ \\ \hline
$1$ & $1$ & $1$ & $0$ & $4$ \\ \hline
$1$ & $2$ & $0$ & $0$ & $3$ \\ \hline
$2$ & $0$ & $0$ & $1$ & $4$ \\ \hline
$2$ & $0$ & $1$ & $0$ & $3$ \\ \hline
$2$ & $1$ & $0$ & $0$ & $2$ \\ \hline
$3$ & $0$ & $0$ & $0$ & $1$ \end{tabular}

First, we prove that $A^1_{n,m}$ is equal to $0$. We have that
$$
A^1_{n,m}(t) = e^{i(\omega_m -\omega_n)t} \frac{im_1}{2} \sum_{k+l = m}E( a_n \overline{a_k a_l} )\; ,
$$
As the law of $g_n$ is invariant by rotation, and that the $g_k$ are independent from each other, we get that for all $k,l,n$,
$$
E(g_kg_lg_n) = 0\; .
$$
Indeed, if one of the indexes, for instance $n$, is different from the other ones and their opposites then $E(g_ng_kg_l) = E(g_n)E(g_kg_l)$ and the mean value of $g_n$ is $0$. Otherwise, it is equal to either $E(g_n^3)$, $E(g_n^2\overline{g_n})$, $E((\overline{g_n})^2 g_n)$, $E(\overline{g_n}^3)$, which are all equal to $0$ by invariance of the law by rotation.

Since $a_k$ is equal to $\lambda_k g_n e^{i\omega_k t}$, we have that
$$
E(a_na_k a_l) = 0
$$
hence $A_{n,m}^1(t)$ is equal to $0$ for all $n,m,t$.

The same argument works for $A^3$ since a sum of products of 1 $a$ and 2 $b$ or 2 $a$ and 1 $c$ is a sum of products of 5 $g$, which gives, as $5$ is odd, $A^3_{n,m}(t) = 0$.

Indeed, we can prove by induction that the expectation of any product of an odd number of $g$ is null.

\begin{lemma}Let $p\in \N$. Then, for all $(n_1, \hdots, n_{2p+1}) \in (\Z^*\times \Z)^{2p+1}$, we have
$$
E\Big( \prod_{i=1}^{2p+1} g_{n_i}\Big) = 0 \; .
$$
\end{lemma}

\begin{proof}We proceed by induction, if $p=0$ then by the invariances satisfied by the law of $g_n$, we have $-E(g_n) = E(e^{i\pi} g) = E(g) = 0$.

For bigger $p$, we consider the sets
$$
A_1 = \lbrace i \; |\; n_i=n_1\rbrace \; A_2 = \lbrace i \; |\; n_i = -n_1\rbrace \; , \; A_{3} = \lbrace i \; |\; |n_i| \neq |n_1| \rbrace \; .
$$
If $A_3$ is empty, we use the invariance of the law, that $g_{-n} = \overline g_n$ and the fact that the difference between the cardinals $m_1$ and $m_2$ of $A_1$ and $A_2$, as their sum is odd, can not be $0$, to have :
$$
E\Big( \prod g_{n_i}\Big) = E\Big(g_{n_1}^{m_1}\overline g_{n_1}^{m_2}\Big) = e^{i\theta (m_1 - m_2)}E\Big(g_{n_1}^{m_1}\overline g_{n_1}^{m_2}\Big) = 0\; .
$$
If $A_3$ is not empty, thanks to the independence, we have
$$
E\Big( \prod g_{n_i}\Big) =E\Big( \prod_{i\in A_1\cup A_2} g_{n_i}\Big) E\Big(\prod_{i\in A_3}g_{n_i}\Big)\; .
$$
As either the cardinal of $A_1\cup A_2$ or the one of $A_3$ is odd and strictly less than $2p+1$, we use the induction hypothesis to conclude the proof. \end{proof}

We now compute $A^2_{n,m}(t)$. It involves products of 1 $b$ and 2 $a$. Thus it can be written
$$
A^2_{n,m}(t) =  e^{i(\omega_m -\omega_n)t}   \frac{im_1}{2} \sum_{k+l = m}\Big( E(a_n \overline{a_k b_l})+ E(a_n \overline{b_k a_l}) + E(b_n \overline{a_k a_l})\Big) \; ,
$$
We call $m_2 = E(|g_n|^2)$ and $m_4 = E(|g_n|^4)$.

Let us compute 
$$
E(a_n \overline{a_k b_l})\; 
$$
when $k+l=m$. We replace $b_l$ by its value, that is
$$
b_l= -\frac{l_1}{2} \sum_{j+q=l} e^{i\omega_l t }\frac{e^{i\Delta_l^{j,q}t}-1}{\Delta_l^{j,q}} \lambda_j\lambda_q g_j g_q\; ,
$$
which gives
$$
 E(a_n \overline{a_k b_l}) = -\frac{l_1}{2} \sum_{j+q=l}e^{-i\omega_l t }\frac{e^{-i\Delta_l^{j,q}t}-1}{\Delta_l^{j,q}} e^{i\omega_n t}e^{-i\omega_k t} \overline{\lambda_j\lambda_q \lambda_k} \lambda_n E(g_n \overline{g_k g_j g_q}) \; .
$$
For $E(g_n \overline{g_k g_j g_q})$ not to be $0$, we have to pair the indexes, otherwise the invariance by rotation and independence make it null. Indeed, we have the following lemma.

\begin{lemma}Let $(n_1,\hdots,n_4) \in \Z^*\times \Z$. We have :
$$
E\Big( \prod_{i=4}^4 g_{n_i} \Big) = \left \lbrace{ \begin{tabular}{lll}
$E(|g_n|^4)$ & \mbox{ if } $\exists \sigma \in S^4$ \mbox{ such that } $n_{\sigma(1)}=n_{\sigma(2)} = - n_{\sigma(3)} = -n_{\sigma(4)}$ \\
$E(|g_n|^2)^2$ & \mbox{ if } $\exists \sigma  \in S^4$ \mbox{ such that } $n_{\sigma(1)} = -n_{\sigma(3)}$ \mbox{ and } $n_\sigma(2) = -n_{\sigma(4)}$\\
$0$ & \mbox{ otherwise } \end{tabular}}\right.
$$
where $S^4$ is the set of permutations of $\{1,2,3,4\}$.
\end{lemma}

\begin{proof} Let $A = \{ |n_i| \; |\; i=1, \hdots, 4\}$ and for all $n \in A$, let $m_1(n)$ be the cardinal of the set $\{i\; |\; n_i = n\}$ and $m_2(n)$ be the cardinal of $\{i\; |\; n_i = -n\}$. 

If $A$ has $4$ elements, then for all $n\in A$, $m_1(n) + m_1(n) = 1$, hence by independence and invariance by rotation
$$
E\Big( \prod_{i=4}^4 g_{n_i} \Big) =  \prod_{i=4}^4 E\Big(g_{n_i} \Big) =0 \; .
$$

If $A$ has $3$ elements, then there exists $n\in A$ such that $m_1(n) + m_2(n) = 0$. Calling $i_0$ the unique index such that $n_{i_0}= \pm n$, we get
$$
E\Big( \prod_{i=4}^4 g_{n_i} \Big) =E(g_{\pm n_{i_0}}) E\Big(\prod_{i\neq i_0} g_{n_i}\Big)=0\; .
$$

If $A$ has $2$ elements, that is $A = \{n,\overline n\}$, there is a first case : $m_1(n) = m_2(n) = m_1(\overline n )=m_2(\overline n)$. This is equivalent to the existence of $\sigma  \in S^4$ such that $n_{\sigma(1)} = -n_\sigma(3) = n$ and $n_{\sigma(2)} = -n_{\sigma(4)} = \overline n$. And we have $n \neq \overline n$. In this case,
$$
E\Big( \prod_{i=4}^4 g_{n_i} \Big) = E(|g_n|^2) E(|g_{\overline n}) = E(|g_n|^2)^2 \; .
$$
In the other case, we have
\begin{eqnarray*}
E\Big( \prod_{i=4}^4 g_{n_i} \Big) & = & E(g_n^{m_1(n)}\overline{g_n}^{m_2(n)})E(g_{\overline n}^{m_1(\overline n)}\overline{g_{ \overline n}}^{m_2(\overline n)}) \\
& = & e^{i\theta (m_1(n) - m_2(n))}e^{i\overline \theta (m_1 (\overline n )-m_2(\overline n)) }  E(g_n^{m_1(n)}\overline{g_n}^{m_2(n)})E(g_{\overline n }^{m_1(\overline n)}\overline{g_{ \overline n}}^{m_2(\overline n)}) \\
 &= & 0
\end{eqnarray*}
for all $\theta, \overline \theta$.

If $A$ has only one element, $A=\{n\}$, there is a first case $m_1(n) = m_2(n) = 2$. This is equivalent to the existence of $\sigma  \in S^4$ such that $n_{\sigma(1)} = -n_\sigma(3) =n_{\sigma(2)} = -n_{\sigma(4)} =  n$.  In this case,
$$
E\Big( \prod_{i=4}^4 g_{n_i} \Big) = E(|g_n|^4)\; .
$$
Otherwise, we have
$$
E\Big( \prod_{i=4}^4 g_{n_i} \Big) = E(g_n^{m_1(n)}\overline{g_n}^{m_2(n)}) = e^{i\theta (m_1(n) - m_2(n))}E(g_n^{m_1(n)}\overline{g_n}^{m_2(n)}) =0 \; .
$$ \end{proof}

 We can not pair $n$ with $k$, or $l$ would be $(0,0)$ so we can only pair $n$ with either $j$ or $q$ and $-k$ with the other one. This is possible if and only if $j+q = n-k$, that is, $m=k+l=n$. By symmetry of $j$ and $q$, there is two solutions when $j\neq q$, which is equivalent to $n\neq -k$. In this case, $E(|g_n|^2|g_k|^2) = m_2^2$. When $k=-n$, there are only one solution for $j$ and $q$, $l= 2n$, and $E(|g_n|^2|g_k|^2) = m_4$. We have :
$$
E(a_n \overline{a_k b_l}) = \left \lbrace{\begin{tabular}{ll}
$-l_1 \delta_n^m e^{-i\omega_l t }\frac{e^{-i\Delta_l^{n,-k}t}-1}{\Delta_l^{n,-k}} e^{i\omega_n t}e^{-i\omega_k t}|\lambda_n|^2|\lambda_k|^2 m_2^2 $ & \mbox{ when } $k\neq -n$ \\
$- \frac{l_1}{2} \delta_n^m e^{-i\omega_l t }\frac{e^{-i\Delta_l^{n,-k}t}-1}{\Delta_l^{n,-k}} e^{i\omega_n t}e^{-i\omega_k t}|\lambda_n|^2|\lambda_k|^2m_4 $ & \mbox{ otherwise.} \end{tabular}} \right.
$$
which can be rewritten as, using that $\omega_{-k} = - \omega_k$, and thus $\Delta_l^{n,-k} = - \Delta_n^{k,l}$,
\begin{equation}\label{sum-1}
E(a_n \overline{a_k b_l}) = l_1\delta_n^m \frac{1 -e^{-i\Delta_n^{k,l}t}}{\Delta_n^{k,l}} |\lambda_n|^2|\lambda_k|^2 m_2^2 + \delta_n^m\delta_{k+n}^0n_1 (m_4-2m_2^2) \frac{1 -e^{-i\Delta_n^{-n,2n}t}}{\Delta_n^{-n,2n}} |\lambda_n|^4\; .
\end{equation}
In the same way, we have,
\begin{equation}\label{sum-2}
E(a_n \overline{a_l b_k}) = k_1\delta_n^m \frac{1 -e^{-i\Delta_n^{k,l}t}}{\Delta_n^{k,l}} |\lambda_n|^2|\lambda_l|^2 m_2^2 +\delta_n^m\delta_{l+n}^0 n_1 (m_4-2m_2) \frac{1 -e^{-i\Delta_n^{-n,2n}t}}{\Delta_n^{-n,2n}} |\lambda_n|^4\; .
\end{equation}

Let us compute 
$$
E(b_n \overline{a_k a_l})\; .
$$
By replacing $b_n$ by its value, we have : 
$$
E(b_n \overline{a_k a_l}) = -\frac{n_1}{2} \sum_{j+q=n} e^{i\omega_n t }\frac{e^{i\Delta_n^{j,q}t}-1}{\Delta_n^{k,l}} e^{-i\omega_k t }e^{-i\omega_l t}\lambda_j\lambda_q \overline{\lambda_k\lambda_l}E(g_j g_q \overline{g_kg_l})\; .
$$
To get a non zero mean value, we need to couple $j$ with either $k$ or $l$ and $q$ with the other one. This is possible if and only if $n=j+q =k+l = m$. If $n_1$ and $n_2$ are even and $k=l=n/2$ then there is only one solution for $j$ and $q$, and  $E(|g_k|^2|g_l|^2) = m_4$. Otherwise, by symmetry, there are 2 solutions, which gives : 
\begin{equation}\label{sum-3}
E(b_n \overline{a_k a_l}) = -n_1 \delta_n^m \frac{1 - e^{i\Delta_n^{k,l}t}}{\Delta_n^{k,l}} |\lambda_k|^2|\lambda_l|^2 m_2^2 - \delta_n^m \delta_{n/2 \in \N^2} (m_4-2m_2^2) \frac{n_1}{2} \frac{1 - e^{i\Delta_n^{n/2,n/2}t}}{\Delta_n^{n/2,n/2}} |\lambda_k|^4\; .
\end{equation}

Summing up \eqref{sum-1}, \eqref{sum-2}, and \eqref{sum-3}, we get an expression of $A_{n,m}^2(t)$ : 
\begin{eqnarray*}
A_{n,m}^2(t) & = & \delta_n^m m_2^2 \frac{in_1}{2} \sum_{k +l = n} \frac{1 -e^{-i\Delta_n^{k,l}t}}{\Delta_n^{k,l}} \Big( k_1 |\lambda_n|^2|\lambda_l|^2 + l_1 |\lambda_n|^2|\lambda_k|^2 - n_1 |\lambda_k|^2|\lambda_l|^2\Big) + \\
 & &\delta_n^m (m_4-2 m_2^2) \frac{in_1}{2}\Big( 2n_1\frac{1 -e^{-i\Delta_n^{-n,2n}t}}{\Delta_n^{-n,2n}} |\lambda_n|^4 - \delta_{n/2 \in \N^2}\frac{n_1}{2} \frac{1 - e^{i\Delta_n^{n/2,n/2}t}}{\Delta_n^{n/2,n/2}} |\lambda_{n/2}|^4\Big) \; .\end{eqnarray*}
Taking twice the real part of $A^2_{n,n}(t)$ gives the expression of $G_n(t)$.

For the bound on $R$, we remark that 
$$
|m_1 \sum_{k+l=m} E(\alpha_k \beta_l \gamma_n) | \leq 2 |n|^{-s} E\left( \|\alpha\|_{H^s}\|\beta\|_{H^s}\|\gamma\|_{H^s} \right) \leq C|n|^{-s}\|\alpha\|_{L^\infty,H^s}\|\beta\|_{L^\infty,H^s}\|\gamma\|_{L^\infty,H^s} \; .
$$
Hence, using the table, we get that
\begin{eqnarray*}
|A^4_{n,m} (t)| & \lesssim & |n|^{-s}\Big( \|a\|_{L^\infty,H^s}^2\|d\|_{L^\infty,H^s}+\|a\|_{L^\infty,H^s}\|b\|_{L^\infty,H^s}\|c\|_{L^\infty,H^s}+\|b\|_{L^\infty,H^s}^3\Big)\\
|A^5_{n,m} (t)| & \lesssim & |n|^{-s}\Big( \|a\|_{L^\infty,H^s}\|b\|_{L^\infty,H^s}\|d\|_{L^\infty,H^s}+\|a\|_{L^\infty,H^s}\|c\|_{L^\infty,H^s}^2+\|b\|_{L^\infty,H^s}^2\|c\|_{L^\infty,H^s}\Big)\\
|A^6_{n,m} (t)| & \lesssim & |n|^{-s}\Big( \|a\|_{L^\infty,H^s}\|c\|_{L^\infty,H^s}\|d\|_{L^\infty,H^s}+\|b\|_{L^\infty,H^s}^2\|d\|_{L^\infty,H^s}+\|b\|_{L^\infty,H^s}\|c\|_{L^\infty,H^s}^2\Big)\\
|A^7_{n,m} (t)| & \lesssim & |n|^{-s}\Big( \|a\|_{L^\infty,H^s}\|d\|_{L^\infty,H^s}^2+\|b\|_{L^\infty,H^s}\|c\|_{L^\infty,H^s}\|d\|_{L^\infty,H^s}+\|c\|_{L^\infty,H^s}^3\Big)\\
|A^8_{n,m} (t)| & \lesssim & |n|^{-s}\Big( \|b\|_{L^\infty,H^s}\|d\|_{L^\infty,H^s}^2+\|c\|_{L^\infty,H^s}^2\|d\|_{L^\infty,H^s}\Big)\\
|A^9_{n,m} (t)| & \lesssim & |n|^{-s} \|c\|_{L^\infty,H^s}\|d\|_{L^\infty,H^s}^2\\
|A^{10}_{n,m} (t)| & \lesssim & |n|^{-s}\|d\|_{L^\infty,H^s}^3
\end{eqnarray*}
and use the symmetry to bound $A^j_{m,n}$ and then $R$.
\end{proof}

\subsection{Expansion of the moments of order 3}\label{sub-momthr}

We sum up here the computations related to the development of $E(u_nu_mu_p)$. 

In this subsection, we expand $u$ only up to order 2 hence, we write $e= c+ \varepsilon d$ and $u = a+ \varepsilon b + \varepsilon^2 e$. In the sequel, the norm of $e$ will be bounded by $\|c\| + \varepsilon \|d\|$.

\begin{lemma}The expansion of $\partial_t E(e^{-i(\omega_n + \omega_m + \omega_p)t}u_nu_mu_p)$ is written : 
$$
\partial_t E(e^{-i(\omega_n + \omega_m + \omega_p)t}u_nu_mu_p) =\varepsilon  \delta_{n+m+p}^0  H_{n,m,p}(t) + \varepsilon^3 R(\varepsilon, n,m,p,t)
$$
with
\begin{eqnarray*}
H_{n,m,p}(t)  & = &  ie^{-i(\omega_n + \omega_m + \omega_p)t} \left( E(|g_n|^2) \Big( n_1 |\lambda_m|^2|\lambda_p|^2 + m_1 |\lambda_p|^2|\lambda_n|^2 + p_1 |\lambda_n|^2|\lambda_m|^2 \Big)+ \right. \\
 & & \frac{1}{2}\left. (E(|g_n|^4 - 2 E(|g_n|^2)^2) \Big( \delta_m^p n_1 |\lambda_m|^4 + \delta_p^n m_1 |\lambda_p|^4 + \delta_n^m p_1 |\lambda_n|^4 \Big) \right)
\end{eqnarray*}
and
\begin{eqnarray*}
R(\varepsilon,n,m,p,t) & \leq & C\max( |m|^{-s}|p|^{-s},|p|^{-s}|n|^{-s},|n|^{-s}|m|^{-s})  \\
 & & \Big( \; \Big( \|a\|_{L^\infty,H^s}^2\|b\|_{L^\infty,H^s}^2+\|a\|_{L^\infty,H^s}^3\|e\|_{L^\infty,H^s}\Big) +  \\
 & & \varepsilon \Big( \|a\|_{L^\infty,H^s}\|b\|_{L^\infty,H^s}^3+\|a\|_{L^\infty,H^s}^2\|b\|_{L^\infty,H^s}\|e\|_{L^\infty,H^s}\Big)  +  \\
 & & \varepsilon^2 \Big(\|b\|_{L^\infty,H^s}^4+ \|a\|_{L^\infty,H^s}\|b\|_{L^\infty,H^s}^2\|e\|_{L^\infty,H^s}+\|a\|_{L^\infty,H^s}^2\|e\|_{L^\infty,H^s}^2\Big) +  \\
 & & \varepsilon^3 \Big( \|b\|_{L^\infty,H^s}^3\|e\|_{L^\infty,H^s}^1+\|a\|_{L^\infty,H^s}\|b\|_{L^\infty,H^s}\|e\|_{L^\infty,H^s}^2\Big) +  \\
 & & \varepsilon^4 \Big( \|b\|_{L^\infty,H^s}^2\|e\|_{L^\infty,H^s}^2+\|a\|_{L^\infty,H^s}\|e\|_{L^\infty,H^s}^3\Big) +  \\
 & &  \varepsilon^5 \|b\|_{L^\infty,H^s}\|e\|_{L^\infty,H^s}^3  + \\
 & &  \varepsilon^6 \|e\|_{L^\infty,H^s}^4 \Big) \; .
\end{eqnarray*}
\end{lemma}

\begin{proof} We start by writing $\partial_t E(e^{-i(\omega_n + \omega_m + \omega_p)t}u_nu_mu_p)$ as
$$
\partial_t E(e^{-i(\omega_n + \omega_m + \omega_p)t}u_nu_mu_p) = A_{n,m,p}(t) +  A_{m,p,n}(t) + A_{p,n,m}(t)
$$
with
$$
A_{n,m,p}(t) = E(\partial_t(e^{-i\omega_n t}u_n) e^{-i(\omega_m+\omega_p)t}u_mu_p)
$$
as in the previous subsection.

By replacing $\partial_t(e^{-i\omega_n t}u_n)$ by its expression we get :
$$
A_{n,m,p}(t) = e^{-i(\omega_n+\omega_m+\omega_p)t}\varepsilon \frac{in_1}{2}\sum_{k+l=n}E(u_ku_lu_mu_p)\; .
$$

We then expand $A_{n,m,p}(t)$ in $\varepsilon$ as
$$
A_{n,m,p}(t)  = \sum_{j=1}^9 \varepsilon^j A_{n,m,p}^j(t)
$$
where $A^j_{n,m,p}(t)$ is of the form
$$
e^{-i(\omega_n+\omega_m+\omega_p)t}\frac{in_1}{2}\sum_{k+l=n}E(\alpha_k\beta_l\gamma_m \eta_p)
$$
with $\alpha,\beta,\gamma,\eta$ replaced by $a$, $b$, or $e$ and $j$ is equal to 1 plus twice the number of occurrences of $e$ plus the number of occurrences of $b$.

Let us compute $A^1_{n,m,p}(t)$. In $A^1$, $\alpha = \beta= \gamma = \eta =a$ hence
$$
A^1_{n,m,p}(t) = e^{i\Delta_n^{k,l}} \frac{in_1}{2} \sum_{k+l = n} \lambda_k\lambda_l\lambda_m\lambda_p E(g_kg_lg_mg_p) \; .
$$
For the mean value not to be $0$, we need to pair $k$ with $-m$ or $-p$ and $l$ with the other one. This is possible if and only if $n+m+p = 0$. If $m=p$ then there is only one solution, otherwise there are 2, which yields
$$
A^1_{n,m,p}(t) = \left \lbrace{\begin{tabular}{ll}
$in_1 e^{-i(\omega_n + \omega_m + \omega_p)t} |\lambda_m|^2|\lambda_p|^2E(|g_n|^2)$ & \mbox{ if }$m\neq p, n+m+p=0$\\
$i\frac{n_1}{2} e^{-i(\omega_n + \omega_m + \omega_p)t} |\lambda_m|^4 E(|g_n|^4)$ & \mbox{ if } $m=p,n+m+p=0$\end{tabular}}\right. 
$$
We can sum this up as : 
$$
A_{n,m,p}(t) = \delta_{n+m+p}^0  e^{-i(\omega_n + \omega_m + \omega_p)t} \left( in_1 |\lambda_m|^2|\lambda_p|^2E(|g_n|^2) + (E(|g_n|^4)-2E(|g_n|^2)^2) \delta_m^p \frac{in_1}{2} |\lambda_m|^4\right) \; .
$$
By summing over the cyclic permutations over $n,m,p$ we get the result.

Let us compute $A^2$. This part of the development is obtained by replacing either $\alpha,\beta,\gamma$, or $\eta$ by $b$ and the other ones by $a$. By replacing $b$ by its expression in terms of $a$, we get that $A^2$ is a sum which involves products of 5 occurrences of $a$, which means that we have to take the mean value of a product of 5 $g$s. But, since $5$ is odd, the $g_n$ are independent from each other and their law is invariant by rotation, any product of 5 $g$ has a null mean value. In the end, we have that :
$$
A^2_{n,m,p}(t) = 0 \; .
$$

Let us bound $A^j_{n,m,p}(t)$ for $j\geq 3$. We remark that the sum
$$
e^{-i(\omega_n+\omega_m+\omega_p)t}\frac{in_1}{2}\sum_{k+l=n}E(\alpha_k\beta_l\gamma_m \eta_p)
$$
is bounded by
$$
C |m|^{-s}|p|^{-s}E\Big( \|\alpha\|_{H^s}\|\beta\|_{H^s}\|\gamma\|_{H^s}\|\eta\|_{H^s} \Big) \leq C |m|^{-s}|p|^{-s}\|\alpha\|_{L^\infty,H^s}\|\beta\|_{_{L^\infty,H^s}}\|\gamma\|_{L^\infty,H^s}\|\eta\|_{L^\infty,H^s}
$$
where the $L^\infty$ norm corresponds to the probability space $\Omega,\mathcal A,\mathbb P$.

The following table gives $j$ the order in $\varepsilon$ in function of the number of occurrences of $a$, $b$, and $e$.

\begin{tabular}{c|c|c|c}
Occurrences of $a$ & Occurrences of $b$ & Occurrences of $e$ & Order in $\varepsilon$ \\ \hline
$0$ & $0$ & $4$ & $9$ \\ \hline
$0$ & $1$ & $3$ & $8$ \\ \hline
$0$ & $2$ & $2$ & $7$ \\ \hline
$0$ & $3$ & $1$ & $6$ \\ \hline
$0$ & $4$ & $0$ & $5$ \\ \hline
$1$ & $0$ & $3$ & $7$ \\ \hline
$1$ & $1$ & $2$ & $6$ \\ \hline
$1$ & $2$ & $1$ & $5$ \\ \hline
$1$ & $3$ & $0$ & $4$ \\ \hline
$2$ & $0$ & $2$ & $5$ \\ \hline
$2$ & $1$ & $1$ & $4$ \\ \hline
$2$ & $2$ & $0$ & $3$ \\ \hline
$3$ & $0$ & $1$ & $3$ \\ \hline
$3$ & $1$ & $0$ & $2$ \\ \hline
$4$ & $0$ & $0$ & $1$ \end{tabular}

In $A^3$, there are 2 $b$ and 2 $a$ or 1 $e$ and 3 $a$, hence
$$
|A^3_{n,m,p}(t)|\leq C |m|^{-s}|p|^{-s} \Big( \|b\|_{L^\infty,H^s}^2\|a\|_{L^\infty,H^s}^2+\|e\|_{L^\infty,H^s}\|a\|_{L^\infty,H^s}^3\Big)\; .
$$

In $A^4$, there are 3 $b$ and 1 $a$ or 1 $e$, 1 $b$ and 2 $a$, hence
$$
|A^4_{n,m,p}(t)|\leq C|m|^{-s}|p|^{-s} \Big( \|b\|_{L^\infty,H^s}^3\|a\|_{L^\infty,H^s}+\|e\|_{L^\infty,H^s}\|a\|_{L^\infty,H^s}^2\|b\|_{L^\infty,H^s}\Big)\; .
$$

In $A^5$, there are 4 $b$, or 1 $a$, 2 $b$, 1 $e$, or 2 $a$, 2 $e$  hence
$$
|A^5_{n,m,p}(t)|\leq C |m|^{-s}|p|^{-s}\Big(\|b\|_{L^\infty,H^s}^4+ \|b\|_{L^\infty,H^s}^2\|a\|_{L^\infty,H^s}\|e\|_{L^\infty,H^s}+\|e\|_{L^\infty,H^s}^2\|a\|_{L^\infty,H^s}^2\Big)\; .
$$

In $A^6$, there are 3 $b$ and 1 $e$ or 1 $a$, 1 $b$ and 2 $e$, hence
$$
|A^6_{n,m,p}(t)|\leq C |m|^{-s}|p|^{-s}\Big( \|b\|_{L^\infty,H^s}^3\|e\|_{L^\infty,H^s}+\|a\|_{L^\infty,H^s}\|e\|_{L^\infty,H^s}^2\|b\|_{L^\infty,H^s}\Big)\; .
$$

In $A^7$, there are 2 $b$ and 2 $e$ or 1 $a$ and 3 $e$, hence
$$
|A^7_{n,m,p}(t)|\leq C |m|^{-s}|p|^{-s}\Big( \|b\|_{L^\infty,H^s}^2\|e\|_{L^\infty,H^s}^2+\|a\|_{L^\infty,H^s}\|e\|_{L^\infty,H^s}^3\Big)\; .
$$

In $A^8$, there are 1 $b$ and 3 $e$, hence
$$
|A^8_{n,m,p}(t)|\leq C |m|^{-s}|p|^{-s}\|b\|_{L^\infty,H^s}\|e\|_{L^\infty,H^s}^3\; .
$$

In $A^9$, there are 4 $e$, hence
$$
|A^9_{n,m,p}(t)|\leq C |m|^{-s}|p|^{-s}\|e\|_{L^\infty,H^s}^4\; .
$$

We then remark that $A^j_{n,m,p}(t)$ is bounded by a constant $A^j(t)$ independent form $n$,$m$, or $p$, times $|m|^{-s}|p|^{-s}$ to get
$$
|A^j_{n,m,p}(t)| + |A^j_{m,p,n}(t)| + |A^j_{p,n,m}(t)| \leq A^j(t) \max(|m|^{-s}|p|^{-s}, |p|^{-s}|n|^{-s}, |n|^{-s}|m|^{-s}) 
$$
and then we sum over $j$ to get the bound on $R$.
\end{proof}

\section{Normal forms}\label{sec-norfor}

In this section, we first write the expansion of the solution of KP-II in terms of normal forms, the second subsection is dedicated to proving that for suitable times and small enough non linearities, we can retrieve $d$ in terms of its normal form version, the third one deals with bounds on $d$ (this is where we perform the contraction argument).

\subsection{Rewriting KP-II using multi-linear maps}\label{sub-rewmul}
In this subsection, we modify KP-II in order to bound $d(t)$ in $L^\infty, H^s(\T^2)$, $s>1$.

First, with the definition of $d$ as $\frac{u -(a+\varepsilon b + \varepsilon^2 c)}{\varepsilon^3} $, with $a = U(t)u_0$, $b$ the first Picard interaction, and $c$ the second, we recall that $u$ solves KP-II with initial datum $u_0$ if and only if $d$ solves \eqref{eqond}, which we recall here,
$$
\partial_t d - Ld + \frac{1}{2}\partial_x \Big( b^2 + 2ac + \varepsilon (2ad + 2bc) + \varepsilon^2 (2bd+c^2) + 2\varepsilon^3 cd+\varepsilon^4 d^2 \Big) = 0
$$
with initial datum $0$.

\begin{definition} We denote by $S$ the bilinear map defined in terms of Fourier coefficients by : 
$$
S(u,v)_n = \frac{in_1}{2} \sum_{k+l=n}\frac{u_kv_l}{i\Delta_n^{k,l}}\; .
$$
\end{definition}

\begin{proposition}If $u$ solves KP-II then $v$ defined as
$$
v= u+\varepsilon S(u,u)
$$
solves 
$$
\partial_t v - L v  = \varepsilon^2 F(u,u,u)
$$
with $F$ a trilinear map defined in terms of its Fourier coefficients as : 
$$
F(\alpha,\beta,\gamma)_n= \frac{n_1}{2}\sum_{j+k+l=n} \frac{j_1+k_1}{i\Delta_n^{j+k,l}} \alpha_j \beta_k \gamma_l 
$$
and with initial datum $v(t=0) = u_0 + \varepsilon S(u_0,u_0)$.
\end{proposition}

\begin{proof}The derivatives with regard to time of the Fourier coefficients of $u$ are given by : 
$$
\dot u_n = i\omega_n u_n - \frac{in_1 \varepsilon}{2} \sum_{k+l=n}u_ku_l \; .
$$
Besides, the Fourier coefficients of $v$ are equal by definition to :
$$
v_n = u_n + \frac{\varepsilon in_1}{2} \sum_{k+l = n} \frac{u_ku_l}{i\Delta_n^{k,l}} \; .
$$
Hence, the derivative of $v_n$ is equal to 
$$
\dot v_n = \dot u_n + \frac{\varepsilon in_1}{2} \sum_{k+l=n} \frac{\dot u_k u_l +u_k\dot u_l }{i\Delta_n^{k,l}} \; .
$$
By replacing $\dot u_n$ by its value, and using the symmetry over $k$ and $l$ in the non linearity, we get :
\begin{eqnarray*}
\dot v_n  = & i\omega_n u_n - \frac{in_1 \varepsilon}{2} \sum_{k+l=n}u_ku_l + \frac{\varepsilon in_1}{2} \sum_{k+l=n} \frac{i\omega_k u_k u_l +u_ki\omega_l u_l }{i\Delta_n^{k,l}}\\
 & - in_1 \varepsilon \sum_{q+l = n}\frac{u_l}{i\Delta_n^{q,l}}\frac{iq_1 \varepsilon}{2} \sum_{j+k = q} u_ju_k \; ,
\end{eqnarray*}
then as $\Delta_n^{k,l} = \omega_k+ \omega_l -\omega_n $, we have
$$
-1 + \frac{i\omega_k + i\omega_l}{i\Delta_n^{k,l}} = \frac{i\omega_n}{i\Delta_n^{k,l}} \; ,
$$
hence we get that terms in $\varepsilon$ are equal to 
$$
\frac{in_1}{2}\varepsilon \sum_{k+l=n} u_k u_l \frac{i\omega_n}{i\Delta_n^{k,l}} = i\omega_n \varepsilon S(u,u)_n
$$
which is $i\omega_n (v_n - u_n)$. The term in $\varepsilon^2$ is equal to
$$
-in_1 \sum_{j+k+l=n} \frac{i (j_1+k_1)}{2}\frac{u_ju_ku_l}{i\Delta_n^{j+k,l}} = F(u,u,u)_n \; .
$$
\end{proof}

Let us write $b$ and $c$ in terms of $S$ and $F$.

\begin{lemma}The first order in $\varepsilon$ of $u$, that is $b$ is equal to
$$
-S(a,a) + U(t)(u_0,u_0) \; .
$$
\end{lemma}

\begin{proof}The Fourier coefficients of $b$ are given by
$$
b_n(t) = -\frac{n_1}{2} \sum_{k+l=n} e^{i\omega_n t }\frac{e^{i\Delta_n^{k,l}t}-1}{\Delta_n^{k,l}} \lambda_k\lambda_l g_k g_l\; .
$$
By dividing $b_n$ in two sums, we get :
$$b_n(t) = -\frac{n_1}{2} \sum_{k+l=n} \frac{1}{\Delta_n^{k,l}} e^{i\omega_k t}\lambda_ke^{i\omega_l t}\lambda_l g_k g_l
 + e^{i\omega_n t }\frac{n_1}{2} \sum_{k+l=n} \frac{1}{\Delta_n^{k,l}} \lambda_k\lambda_l g_k g_l\; .
$$
The first sum is the $n$th Fourier coefficient of $-S(a,a)$, the second is the one of $U(t)S(u_0,u_0)$.\end{proof}

\begin{lemma}The second order in $\varepsilon$ of $u$, that is $c$, equal to
$$
-2 S(a,b) + \int_{0}^t U(t-t') F(a(t'),a(t'),a(t'))dt' \; .
$$
\end{lemma}

\begin{proof} To prove this lemma, we need to do two remarks. The first one is that we have the following formula :
$$
LS(\alpha,\beta) - S(L\alpha,\beta) - S(\alpha,L\beta) = -\frac{1}{2} \partial_x (\alpha \beta)\; .
$$
Indeed, $L$ is the Fourier multiplier by $i\omega_n$, thus in terms of Fourier coefficients, we get
$$
\Big[LS(\alpha,\beta) - S(L\alpha,\beta) - S(\alpha,L\beta)\Big]_n = \frac{in_1}{2} \sum_{k+l = n} \frac{i\omega_n (\alpha_k\beta_l) - (i\omega_k \alpha_k) \beta_l - \alpha_k (i\omega_l \beta_l)}{i\Delta_n^{k,l}}
$$
and we then recall that $\Delta_n^{k,l} = \omega_k +\omega_l -\omega_n$ to get the simplification
$$
\Big[LS(\alpha,\beta) - S(L\alpha,\beta) - S(\alpha,L\beta)\Big]_n = -\frac{in_1}{2} \sum_{k+l = n} \alpha_k \beta_l = \Big[ -\frac{1}{2}\partial_x (\alpha\beta) \Big]_n \; .
$$

The second remark is that
$$
F(\alpha,\beta,\gamma) = - S(\gamma, \partial_x (\alpha\beta)) \; .
$$
Indeed, in terms of Fourier coefficients
$$
\Big[ S(\gamma, \partial_x (\alpha\beta))\Big]_n = \frac{in_1}{2} \sum_{k+l=n} \frac{\gamma_k il_1 \sum_{j+q=l} \alpha_j\beta_q}{i\Delta_n^{k,l}}
$$
and suppressing the intermediary use of $l$,
$$
\Big[ S(\gamma, \partial_x (\alpha\beta))\Big]_n = \frac{in_1}{2} \sum_{j+k+q=n} \frac{j_1+q_1}{\Delta_n^{k,j+q}}\gamma_k \alpha_j\beta_q
$$
which is equal to minus the $n$-th Fourier coefficient of $F(\alpha,\beta,\gamma)$. 

We can now prove the lemma. We call
$$
f = \int_{0}^t U(t-t') F(a(t'),a(t'),a(t'))dt' \; ,
$$
it satisfies the equation
$$
\partial_t f - Lf = F(a,a,a)
$$
with initial datum $0$. Besides $S(a,b)$ satisfies another equation, we have
$$
\partial_t S(a,b) = S(\partial_t a,b) + S(a,\partial_t b)
$$
and using the equations satisfied by $a$ and $b$, we get
$$
\partial_t S(a,b) = S(La,b) + S(a,Lb - \frac{1}{2}\partial_x a^2)\; .
$$
Using our two remarks, we get
$$
\partial_t S(a,b) = LS(a,b) + \frac{1}{2}\partial_x (ab) + \frac{1}{2}F(a,a,a)\; .
$$
Thus, $f-2S(a,b)$ satisfies the equation 
$$
\partial_t (f-2S(a,b)) - L (f-2S(a,b)) = -\partial_x (ab)
$$
with initial datum $0$ as $b(t=0) = 0$. The solution of this equation being unique in the spaces we consider, we have $c= f -2S(a,b)$. \end{proof}

Finally, we expand $v$ until order 3 in $\varepsilon$, we have
$$
v=a + \varepsilon U(t) S(u_0,u_0) + \varepsilon^2 f + \varepsilon^3 w(\varepsilon) \; .
$$
We recall that $a$ solves $\partial_t a -La = 0$ with initial datum $u_0$, that $U(t)S(u_0,u_0)$ solves the same equation with initial datum $S(u_0,u_0)$, and that $f$ solves
$$
\partial_t f -Lf = F(a,a,a)
$$
with initial datum $0$. Hence, we get that $w$ is the solution of
$$
\partial_t w-L w = \frac{F(u,u,u)-F(a,a,a)}{\varepsilon} 
$$
with initial datum $0$.

By expanding $u$ in the expression of $v$, we get
\begin{eqnarray*}
v & = & a + \varepsilon \Big( b+S(a,a) \Big) + \varepsilon^2 \Big(c + 2S(a,b) \Big) +  \varepsilon^3 \Big( d + S(b,b) +2S(a,c)\\
& &  + 2\varepsilon ( S(b,c) + S(a,d)) +\varepsilon^2 (S(c,c) +2S(b,d)) +2\varepsilon^3 S(c,d) +\varepsilon^4 S(d,d)\Big) \; .
\end{eqnarray*}
Using our lemmas, that is $b+S(a,a) = U(t)S(u_0,u_0)$ and $c+2S(a,b) = f$, and by identification, we get
$$
w = d + S(b,b) +2S(a,c) + 2\varepsilon ( S(b,c) + S(a,d)) +\varepsilon^2 (S(c,c) +2S(b,d)) +2\varepsilon^3 S(c,d) +\varepsilon^4 S(d,d) \; .
$$

\subsection{Local inversion}\label{sub-locinv}

We prove in this subsection, that as long as $w$ is small enough, ten $d$ can be retrieved in terms of $w$.

\begin{proposition}Let $s>\frac{1}{2}$ then $S$ is a continuous bilinear map from $(L^\infty,H^s)^2$ to $L^\infty,H^s$.

With $s>1$, $F$ is a continuous trilinear map from $(L^\infty,H^s)^3$ to $L^\infty,H^s$.\end{proposition}

\begin{proof} This proof is similar to the corresponding ones in \cite{Tlon}. Let $u,v,w \in L^\infty, H^s$. We first suppose that $s>1/2$. Let $\alpha = S(u,v)$, then $\alpha_n$ satisfies : 
$$
2\alpha_n = in_1\sum_{k+l = n} \frac{u_kv_l}{i\Delta_n^{k,l}} \; .
$$

We recall that $|\Delta_n^{k,l}| \geq 3|k_1 l_1 n_1|$ as long as $k+l=n$. Hence, using that $|n|^s \leq C_s(|k|^s+|l|^s)$ ,
$$
|n|^{s}|\alpha_n|\leq C_s \sum_k \frac{|n-k|^s|u_{n-k}|\; |v_k| }{|k_1|} + C_s \sum_k \frac{|n-k|^s|v_{n-k}|\; |u_k| }{|k_1|} \; .
$$
We take this quantity to the square and we sum it over $n$ to get
$$
\|\alpha\|_{H^s}^2 \leq C_s^2 \left( \sum_{k,l} \frac{v_k v_l}{|k_1 l_1|} \sum_n |n-k|^s|u_{n-k}|\; |n-l|^s|u_{n-l}| + \sum_{k,l} \frac{u_k u_l}{|k_1 l_1|} \sum_n |n-k|^s|v_{n-k}|\; |n-l|^s|v_{n-l}| \right) .
$$
Using a Cauchy-Schwartz inequality on the sums over $n$ we get : 
$$
\|\alpha\|_{H^s} \leq C_s \left( \|u\|_{H^s} \sum_k \frac{|v_k|}{|k_1|} + \|v\|_{H^s} \sum_k \frac{|u_k|}{|k_1|}\right)\; ,
$$
and then on the sums over $k$,
$$
\sum_k \frac{|v_k|}{|k_1|} \leq \|v\|_{H^s} \sqrt{\sum_k \frac{1}{|k_1|^2\;|k|^{2s}}} \; .
$$
The series converges as $s> 1/2$, hence
$$
\|S(u,v)\|_{H^s} \leq C_s \|u\|_{H^s}\|v\|_{H^s} \;, 
$$
and by taking its $L^\infty$ norm in the probability space,
$$
\|S(u,v)\|_{L^\infty,H^s} \leq C_s \|u\|_{L^\infty,H^s}\|v\|_{L^\infty,H^s} \; .
$$

We suppose now that $s>1$. Let $\beta = F(u,v,w)$, then $\beta_n$ satisfies
$$
2\beta_n = -n_1 \sum_{j+k+l=n} (j_1+k_1) \frac{u_j v_k w_l}{i\Delta_n^{j+k,l}} \; .
$$
By using the same inequalities as previously on $\Delta_{n}^{k,l}$ and $|n|^s$, we get
$$
\frac{|n_1|\;|j_1+k_1|}{|\Delta_n^{k,l}|} \leq 3\; ,
$$
hence
$$
|n|^s |\beta_n| \leq C_s \sum_{j+k+l=n} (|j|^s+|k|^s+|l|^s) |u_j|\; |v_k|\; |w_l|\; .
$$
Using Cauchy Schwartz inequalities on the sum over $n$ and symmetries of this expression, we get
$$
\|\beta\|_{H^s} \leq C_s \left( \|u\|_{H^s} \Big( \sum_k |v_k|\Big) \Big(\sum_l |w_l|\Big) + \|v\|_{H^s} \Big( \sum_k |w_k|\Big) \Big(\sum_l |u_l|\Big)+ \|w\|_{H^s} \Big( \sum_k |u_k|\Big) \Big(\sum_l |v_l|\Big)\right)\; ,
$$
and then
$$
\sum_k |v_k| \leq \|v\|_{H^s}\sqrt{\sum_k |k|^{-2s}}
$$
and the series converges as $k$ is of dimension $2$ and $s>1$. Therefore,
$$
\|\beta\|_{H^s} \leq C_s \|u\|_{H^s}\|v\|_{H^s}\|w\|_{H^s}
$$
and by taking its $L^\infty$ norm in probability
$$
\|F(u,v,w)\|_{L^\infty,H^s} \leq C_s \|u\|_{L^\infty,H^s}\|v\|_{L^\infty,H^s}\|w\|_{L^\infty,H^s} \; .
$$ \end{proof}

\begin{proposition} Assuming that $s>1$ and that the norm of $u_0$ is 1, there exists a constant $C$ such that for all time $t\in \R$, 
$$
\|b(t)\|_{L^\infty,H^s} \leq C \; ,\; \|c(t)\|_{L^\infty,H^s} \leq C(1+|t|) \; .
$$
\end{proposition}

\begin{proof}We use the descriptions of $b$ and $c$ in terms of $S$ and $F$. We have
$$
b = -S(a,a) + U(t) S(u_0,u_0)
$$
and since $U(t)$ is isometric in $H^s$, it is isometric in $L^\infty,H^s$, we then have, thanks to the continuity of $S$ in $H^s$,
$$
\|b(t)\|_{L^\infty,H^s} \leq C  \|u_0\|_{L^\infty,H^s}^2 \leq C\; .
$$
For $c$, we have
$$
c= -2S(a,b) + f \; ,
$$
we recall that $f$ is given by
$$
f(t) = \int_{0}^t U(t-t') F(a(t'),a(t'),a(t')) dt'
$$
and that $F$ is continuous in $H^s$ thus
$$
\|f\|_{L^\infty,H^s} \leq C|t|
$$
and besides
$$
\|S(a,b)\|_{L^\infty,H^s} \leq C \|a\|_{L^\infty,H^s}\|b\|_{L^\infty,H^s} \leq C
$$
which gives the result.
\end{proof}

\begin{definition}We define $\Lambda_\varepsilon$ by
$$
\Lambda_\varepsilon (d) = d + 2\varepsilon\Big( S(a,d) +  \varepsilon S(b,d) + \varepsilon^2 S(c,d)\Big) + \varepsilon^4 S(d,d) \; .
$$
\end{definition}

\begin{remark} The quantity $\Lambda_\varepsilon (d)$ corresponds to the variation in $d$ of $w$, the term of third order in the expansion of $v$, the solution of
$$
\partial_t v -Lv = F(u,u,u) \; .
$$
More precisely, $w$ is given by
$$
w = S(b,b) + 2S(a,c) + 2\varepsilon S(b,c) + \varepsilon^2 S(c,c) + \Lambda_\varepsilon (d)
$$
and in terms of $v$
$$
v= a + \varepsilon U(t) S(u_0,u_0) + \varepsilon^2 f + \varepsilon^3 w \; .
$$
\end{remark}

\begin{proposition}There exists $\varepsilon_0>0$ and $r_0> 0$, $T_0>0$ such that, with $r = r_0\varepsilon^{-4}$, $T=T_0\varepsilon^{-3}$, for all $\varepsilon \in [0,\varepsilon_0]$, and all $t\in [-T,T]$, $\Lambda_\varepsilon$ is a homeomorphism from the ball $B_r$ of centre $0$ and radius $r$ of the space $L^\infty([-|t|,|t|],L^\infty(\Omega,H^s(\T^2)))$ to the ball of radius $r/2$ of the same space. What is more, for all $g_1,g_2 \in B_{r/2}$
$$
\| \Lambda_\varepsilon^{-1}(g_1) - \Lambda_\varepsilon^{-1}(g_2)\|_{L^\infty,L^\infty,H^s} \leq 2 \|g_1-g_2\|_{L^\infty,L^\infty,H^s} \; .
$$
\end{proposition}

\begin{proof} This is the local inversion theorem on $\Lambda_\varepsilon$ where we keep track of the dependence of the constants on $\varepsilon$. Let $\varphi = Id - \Lambda_\varepsilon$. The differential of $\varphi$ satisfies
$$
d\varphi_{|d} (h) = -2 \varepsilon S(a,h) - 2 \varepsilon^2 S(b,h) - 2\varepsilon^3 S(c,h)-2\varepsilon^4S(d,h) \; .
$$
As the norm of $u_0$ is supposed to be equal to $1$, so is the norm of $a$, and hence, there exists $C_s$ such that the operator norm of $d\varphi_{|c}$ satisfies
$$
\|| d\varphi_{|c} \|| \leq C_s( \varepsilon + \varepsilon^2 + \varepsilon^3 (1+|t|) + \varepsilon^4 \|d\|) \; .
$$
With $\varepsilon$ less than $\varepsilon_0$, $|t|$ less than $T = T_0\varepsilon^{-3}$, $\|d\|$ less than $r=r_0\varepsilon^{-4}$, and choosing the constants $\varepsilon_0$, $r_0$ and $T_0$ small enough, we get that the norm of $d\varphi_{|d}$ is less than $1/2$. We then have that for all $d_1,d_2$ in $B_r$ : 
$$
\|\varphi(d_1) -\varphi(d_2)\|_{L^\infty,H^s} \leq \frac{1}{2}\|d_1-d_2\|_{L^\infty,H^s} \; .
$$
If $g$ is in $B_{r/2}$ we can apply the fixed point theorem on $ \varphi + g$ to get the existence of a unique $d$ in $B_r$ such that
$$
\Lambda_\varepsilon (d) = g\; .
$$
Besides, for all $g_1,g_2 \in B_{r/2}$, with $d_1 =  \Lambda_\varepsilon^{-1}(g_1)$ and $d_2 =  \Lambda_\varepsilon^{-1}(g_2)$, we have that
$$
\|d_1 -d_2\|_{L^\infty,H^s} \leq \|g_1 - g_2\|_{L^\infty,H^s} + \|\varphi(d_1) - \varphi(d_2)\|_{L^\infty,H^s}
$$
which leads to
$$
\|d_1 -d_2\|_{L^\infty,H^s} \leq 2\|g_1 - g_2\|_{L^\infty,H^s} \; .
$$\end{proof}

\subsection{Probabilistic bounds}\label{sub-probbounds}

In this subsection, we perform the contraction argument that gives us the bound on $d$.

\begin{proposition}Let $s>1$, $\alpha \in [1,2]$ and assume $\|u_0\|_{L^\infty,H^s} = 1$. There exists three constants $C_s>0$, $T_1>0$ and $\varepsilon_1 > 0$ such that, with $1+T = T_1\varepsilon^{-\alpha}$, the problem \eqref{eqond} has a unique local solution in $\mathcal C([-T,T],L^\infty,H^s)$ and for all $t\in [-T,T]$, the norm of $d$ satisfies
$$
\|d(t)\|_{L^\infty,H^s} \leq C_s(1+|t|)^{1+\beta} \; ,
$$
with $\beta = 1- \frac{1}{\alpha}$.
\end{proposition}

\begin{proof} If $w$ is the solution of 
$$
\partial_t w + Lw = \frac{F(u,u,u)-F(a,a,a)}{\varepsilon}
$$
with initial datum $0$, then we have that 
$$
d= \Lambda_\varepsilon^{-1} \Big(w- \Big(S(b,b) + 2S(a,c) + 2\varepsilon S(b,c) + \varepsilon^2S(c,c)\Big)\Big)
$$
if $w$, $t$ and $\varepsilon$ are small enough to define $\Lambda_\varepsilon^{-1}$.

Then, $d$ is the solution of the fixed point : 
\begin{eqnarray*}
d (t) =  A(d)(t) & := & \Lambda_\varepsilon^{-1} \Big[ \int_0^t U(t-s)\frac{F(u(s),u(s),u(s)-F(a(s),a(s),a(s))}{\varepsilon})ds - \\
 &   & \Big(S(b,b) + 2S(a,c) + 2\varepsilon S(b,c) + \varepsilon^2S(c,c)\Big)\Big] \; .
\end{eqnarray*}

Let 
$$
\Phi (d) =  \int_0^t U(t-s)\frac{F(u(s),u(s),u(s))-F(a(s),a(s),a(s))}{\varepsilon}ds \; .
$$
As we can do the factorization
$$
F(u,u,u)-F(a,a,a) = F(u-a,u,u) + F(a,u-a,u) - F(a,a,u-a)
$$
we can bound $\frac{F(u,u,u)-F(a,a,a)}{\varepsilon}$ by
$$
C (\|a\|+\varepsilon \|b\| + \varepsilon^2 \|c\| + \varepsilon^3\|d\|)^2(\|b\|+\varepsilon (1+|t|)+ \varepsilon^2 \|d\|)
$$
which gives, with the bounds on $a$, $b$ and $c$,
$$
\|\frac{F(u,u,u)-F(a,a,a)}{\varepsilon}\| \lesssim (1+\varepsilon^2 |t| +\varepsilon^3 \|d\|)^2 (1+\varepsilon |t| +\varepsilon^2 \|d\|)\; .
$$
With $|t|$ less than $T = T_1 \varepsilon^{-\alpha}-1$ and $d$ in the ball of $L^\infty([-|t|,|t|],L^\infty(\Omega,H^s(\T)))$ of radius $C(1+ |t|)^{1+\beta}$, we get that 
\begin{eqnarray*}
\varepsilon^2 |t| & \leq & T_1 \varepsilon^{2-\alpha} \\
\varepsilon^3 \|d\| & \leq &  C  T_1^{1+\beta} \varepsilon^{2(2-\alpha)} \\
\varepsilon |t| & \leq & T_1^{1-\beta} |t|^{\beta}\\
\varepsilon^2 \|d\| & \leq &  C T_1 \varepsilon^{2-\alpha} (1+|t|)^{\beta}
\end{eqnarray*}
hence with $T_1$ small enough, as $\alpha \leq 2$,  $C$ big enough and $\varepsilon \leq 1$, we get
$$
\|\frac{F(u,u,u)-F(a,a,a)}{\varepsilon}\| \leq \frac{C}{4}(1+|t|)^{\beta}
$$
and then integrating over time
$$
\|\phi (d) \| \leq \frac{C}{4}(1+|t|)^{\beta}|t|\; .
$$
Then, we can remark that
$$
\|S(b,b) + 2S(a,c) + 2\varepsilon S(b,c) + \varepsilon^2S(c,c)\| \lesssim 1 + |t| +\varepsilon^2 |t|^2
$$
thus with $C$ big enough, as $\varepsilon^2 |t| \leq T_1 \varepsilon^{2-\alpha}$,
$$
\|S(b,b) + 2S(a,c) + 2\varepsilon S(b,c) + \varepsilon^2S(c,c)\| \leq \frac{C}{4}(1+|t|)^{1+\beta} \; .
$$
Finally, since $\frac{C}{4}(1+|t|)^{1+\beta} \lesssim \varepsilon^{1-2\alpha} \leq \varepsilon^{-4}$ and because $\Lambda_\varepsilon$ is invertible for radius of order $\varepsilon^{-4}$ and times of order $\varepsilon^{-3}$, we can choose $\varepsilon_1$ small enough such that for all $\varepsilon \leq \varepsilon_1$, we can apply the local inversion theorem on $\Lambda_\varepsilon$ and hence have
$$
\|A(d)\|\leq C(1+|t|)^{1+ \beta} \; .
$$

For the contraction, let $d_1$ and $d_2$ be in the previously considered ball and $u_i = a+\varepsilon b + \varepsilon^2 c + \varepsilon^3 d_i$. Since we can apply the local inversion theorem, we have that
$$
\|A(d_1) - A(d_2)\| \leq 2 \|\phi(d_1)  - \phi(d_2)\|\; .
$$
With the same factorization regarding $F$ as previously, we have
$$
\|\phi(d_1)  - \phi(d_2)\| \lesssim \varepsilon^2 |t| (\|u_1\|+\|u_2\|)^2\|d_1-d_2\| \; .
$$
On the ball and for the times we considered, we have $\|u_i\| \lesssim 1$ and $\varepsilon^2 |t| \leq T_1\varepsilon^{2-\alpha}$ hence for $T_1$ small enough, $A$ is a contraction and we can apply the fixed point theorem to get the result.\end{proof}

\begin{remark}The restriction on the norm of the initial datum can be lifted since, replacing $\|u_0\|_{L^\infty,H^s} = 1$ by $\|u_0\|_{L^\infty,H^s} = \mu$ is equivalent to replacing $\varepsilon$ by $\mu \varepsilon$. \end{remark}

\section{Proof of Theorems \ref{th-mainmomtwo},\ref{th-mainmomthree}}\label{sec-prothe}

\subsection{Estimates on the remainder}\label{sub-estrem}

\begin{proposition}The remainder term of the development of the derivative in time of the moments of order 2 of $u$ is bounded, until times $t$ of order $\varepsilon^{-5/3}$ by : 
$$
|R(\varepsilon,n,m,t)| \leq C_s (\min (|n|,|m|))^{-s} (1+|t|)^{7/5}
$$
whereas the remainder term of the development of the derivative in time of the moments of order 3 of $u$ is bounded, up to times of order $\varepsilon^{-2}$ by : 
$$
|R(\varepsilon, n,m,p,t) |\leq C_s (\min(|nm|,|mp|,|pn|))^{-s} (1+|t|) \; .
$$
\end{proposition}

\begin{proof}For the remainder in $\partial_t E(e^{i(\omega_m - \omega_n)t}u_n \overline{u_m})$, we use the estimate on $R$ given in the subsection \ref{sub-momtwo} : 
\begin{eqnarray*}
|R(\varepsilon,n,m,t)| & \leq & C \min(|n|,|m|)^{-s} \Big( \|a\|^2_{L^\infty,H^s}\|d\|_{L^\infty,H^s} + \|a\|_{L^\infty,H^s}\|b\|_{L^\infty,H^s}\|c\|_{L^\infty,H^s}+\|b\|_{L^\infty,H^s}^3\\
 & & \varepsilon (\|a\|_{L^\infty,H^s}\|b\|_{L^\infty,H^s}\|d\|_{L^\infty,H^s}+\|a\|_{L^\infty,H^s}\|c\|_{L^\infty,H^s}^2+\|b\|_{L^\infty,H^s}^2\|c\|_{L^\infty,H^s} ) + \\
 & & \varepsilon^2 (\|a\|_{L^\infty,H^s}\|c\|_{L^\infty,H^s}\|d\|_{L^\infty,H^s}+\|b\|_{L^\infty,H^s}^2\|d\|_{L^\infty,H^s}+\|b\|_{L^\infty,H^s}\|c\|_{L^\infty,H^s}^2) + \\
 & & \varepsilon^3 (\|b\|_{L^\infty,H^s}\|c\|_{L^\infty,H^s}\|d\|_{L^\infty,H^s}+\|c\|_{L^\infty,H^s}^3) + \\
 & & \varepsilon^4 (\|b\|_{L^\infty,H^s}\|d\|_{L^\infty,H^s}^2+\|c\|_{L^\infty,H^s}^2\|d\|_{L^\infty,H^s}) + \\
 & & \varepsilon^5 \|c\|_{L^\infty,H^s}\|d\|_{L^\infty,H^s}^2 + \\
 & & \varepsilon^6 \|d\|_{L^\infty,H^s}^3 \Big)\; ,
\end{eqnarray*}
Then, we use the bounds from the previous section with $\alpha = 5/3$ and $\beta = 2/7$, assuming that $\|u\|_{L^\infty,H^s} = 1$ 
$$
\|a(t)\|_{L^\infty,H^s} = 1 \; ,\; \|b(t)\|_{L^\infty,H^s} \leq C_s\; , \; \|c(t)\|_{L^\infty,H^s} \leq C_s(1+|t|) \; , \;  \|d\|_{L^\infty,H^s} \leq C_s (1+|t|)^{7/5} \; .
$$
By inputting these estimates into the bound of $R$, we get : 
\begin{eqnarray*}
|R(\varepsilon,n,m,t)| & \leq & C \min(|n|,|m|)^{-s} \Big( (1+|t|)^{7/5} +\varepsilon ((1+|t|)^2 +\varepsilon^2 (1+|t|)^{12/5}+  \varepsilon^3 (1+|t|)^3 + \\
 & & \varepsilon^4 (1+|t|)^{17/5} +  \varepsilon^5 (1+|t|)^{19/5} + \varepsilon^6 (1+|t|)^{21/5} \Big)\; .
\end{eqnarray*}
Then, we factorize it by $(1+|t|)^{7/5} $ to get
\begin{eqnarray*}
|R(\varepsilon,n,m,t)| & \leq & C \min(|n|,|m|)^{-s}  (1+|t|)^{7/5}\Big( 1 +\varepsilon ((1+|t|)^{3/5} +\varepsilon^2 (1+|t|)+  \varepsilon^3 (1+|t|)^{8/5} + \\
 & & \varepsilon^4 (1+|t|)^2 +  \varepsilon^5 (1+|t|)^{12/5} + \varepsilon^6 (1+|t|)^{14/5} \Big)\; .
\end{eqnarray*}
Finally, we use the bound $(1+|t|) \lesssim \varepsilon^{-5/3}$ to get
\begin{multline*}
1 +\varepsilon ((1+|t|)^{3/5} +\varepsilon^2 (1+|t|)+  \varepsilon^3 (1+|t|)^{8/5} + \varepsilon^4 (1+|t|)^2\\
 +  \varepsilon^5 (1+|t|)^{12/5} + \varepsilon^6 (1+|t|)^{14/5}  \lesssim 1+ 1 + \varepsilon^{1/3} + \varepsilon^{1/3} + \varepsilon^{2/3} + \varepsilon  + \varepsilon^{4/3}\; .
\end{multline*}
In the end, we have that :
$$
|R(\varepsilon,n,m,t)| \leq C_s (\min (|n|,|m|))^{-s} (1+|t|)^{7/5}\; .
$$
The integral in time gives the same estimate as in Theorem \ref{th-mainmomtwo}.

For the remainder in $\partial_t E (e^{-i(\omega_n + \omega_m + \omega_p)t}u_n u_m u_p)$, we use the estimate on $R$ in the subsection \ref{sub-momthr}:
\begin{eqnarray*}
|R(\varepsilon,n,m,p,t)| & \leq & C\max( |m|^{-s}|p|^{-s},|p|^{-s}|n|^{-s},|n|^{-s}|m|^{-s}) \Big(  \\
 & & \Big( \|b\|_{L^\infty,H^s}^2\|a\|_{L^\infty,H^s}^2+\|e\|_{L^\infty,H^s}\|a\|_{L^\infty,H^s}^3\Big) +  \\
 & & \varepsilon \Big( \|b\|_{L^\infty,H^s}^3\|a\|_{L^\infty,H^s}^1+\|e\|_{L^\infty,H^s}\|a\|_{L^\infty,H^s}^2\|b\|_{L^\infty,H^s}\Big)  +  \\
 & & \varepsilon^2 \Big(\|b\|_{L^\infty,H^s}^4+ \|b\|_{L^\infty,H^s}^2\|a\|_{L^\infty,H^s}\|e\|_{L^\infty,H^s}+\|e\|_{L^\infty,H^s}^2\|a\|_{L^\infty,H^s}^2\Big) +  \\
 & & \varepsilon^3 \Big( \|b\|_{L^\infty,H^s}^3\|e\|_{L^\infty,H^s}^1+\|a\|_{L^\infty,H^s}\|e\|_{L^\infty,H^s}^2\|b\|_{L^\infty,H^s}\Big) +  \\
 & & \varepsilon^4 \Big( \|b\|_{L^\infty,H^s}^2\|e\|_{L^\infty,H^s}^2+\|a\|_{L^\infty,H^s}\|e\|_{L^\infty,H^s}^3\Big) +  \\
 & &  \varepsilon^5 \|b\|_{L^\infty,H^s}\|e\|_{L^\infty,H^s}^3  + \\
 & &  \varepsilon^6 \|e\|_{L^\infty,H^s}^4 \Big) \; .
\end{eqnarray*}
We use the bounds of $a$, $b$, and $e$ in this expression. With the bounds of the previous subsection applied with $\alpha =2$ and thus $\beta = 1/2$, we have
$$
\|a\|_{L^\infty,H^s} = 1 \; , \; \|b\|_{L^\infty,H^s} \leq C_s \; ,
$$
$$
\|e\|_{L^\infty,H^s}\leq \|c\|_{L^\infty,H^s} + \varepsilon \|d\|_{L^\infty,H^s} \leq C(1+|t| + \varepsilon (1+|t|)^{3/2}) \leq C (1+|t|)\; .
$$

We use these estimates in the bound of $R$ :
\begin{eqnarray*}
R(\varepsilon,n,m,p,t) & \leq & C\max( |m|^{-s}|p|^{-s},|p|^{-s}|n|^{-s},|n|^{-s}|m|^{-s}) \left( \Big(1+  (1+|t|)\Big) + \right. \\
 & &\left. \varepsilon \Big( 1+(1+|t|)\Big)  +  \varepsilon^2 \Big(1 +(1+|t|)+(1+|t|)^2\Big) +  \varepsilon^3 \Big( (1+|t|)+(1+|t|)^2\Big) + \right. \\
 & &\left. \varepsilon^4 \Big( (1+|t|)^2+(1+|t|)^3\Big) + \varepsilon^5 (1+|t|)^3+ \varepsilon^6 (1+|t|)^4 \right) \; .
\end{eqnarray*}

Again, we use that the estimates on $e$ are valid only until times of order $\varepsilon^{-2}$, so that we can bound $(1+|t|)\varepsilon^2$ by some constant. As well, we use that $\varepsilon$ is less than $1$ and that the bigger the power on $(1+|t|)$ is, the worse the estimates are, to bound all the terms involved in the remainder by $C_s(1+|t|)$. In the end, we have that : 
$$
|R(\varepsilon, n,m,p,t) |\leq C_s (\min(|nm|,|mp|,|pn|))^{-s} (1+|t|) \; .
$$
The integral in time gives the same estimate as in Theorem \ref{th-mainmomthree}.\end{proof}

\subsection{Estimates on the different terms of the formal expansion}\label{sub-estdif}

In this subsection, we estimate $\int_t G_n$ and $\int_t H_{n,m,p}$.

\begin{proposition} The second order in $ \varepsilon$ of the expansion of $E(u_n\overline{u_m})$ is equal to 
$$
F_{n,m}(t) = \delta_n^m \int_{0}^t G_n(t')dt'
$$
and the quantity 
$$
\sum_{n,m} \sqrt{|n_1m_1|} |nm|^{s}|F_{n,m}(t)| 
$$
is uniformly bounded in time.
\end{proposition}

\begin{proof} The integral of $G_n$ over $t$ is equal to 
\begin{eqnarray*}
\int_{0}^t G_n(t')dt' &= &  - n_1 E(|g_n|^2)^2 \sum_{k +l = n} \frac{\cos(\Delta_n^{k,l}t)-1}{(\Delta_n^{k,l})^2} \Big( k_1 |\lambda_n|^2|\lambda_l|^2 + l_1 |\lambda_n|^2|\lambda_k|^2 - n_1 |\lambda_k|^2|\lambda_l|^2\Big)  \\
  & & -n_1 (E(|g_n|^4)-2 E(|g_n|^2)^2) \Big( 2n_1 \frac{\cos(\Delta_n^{-n,2n}t)-1}{(\Delta_n^{-n,2n})^2} |\lambda_n|^4 - \\
 & & \delta_{n/2 \in \N^2}\frac{n_1}{2} \frac{\cos(\Delta_n^{n/2,n/2}t)-1}{(\Delta_n^{n/2,n/2})^2} |\lambda_{n/2}|^4\Big) \; .
\end{eqnarray*} 
For $n\neq m$, the derivative in time of $e^{i(\omega_m - \omega_n)t} F_{n,m}(t)$ is equal to $0$ and initially $F_{n,m}(0) = 0$, hence at all time $F_{n,m}(t) = 0$. For $n=m$, $\omega_m = \omega_n$, hence we have an explicit formula for the derivative with regard to time of $F_{n,n}$, and besides, $F_{n,n}(t=0) = 0$, therefore
$$
F_{n,m}(t) = \delta_n^m \int_{0}^t G_n(t')dt'\; .
$$
We bound $E(|g_n|^4)$ and $E(|g_n|^2)^2$ by $\|g_n\|_{L^\infty}^4$ and we remark that the terms depending on $E(|g_n|^4)$ can be written like the terms under the sum to get : 
$$
|F_{n,n}(t)| \leq C\|g_n\|_{L^\infty}^4 |n_1| \sum_{k+l = n} (\Delta_n^{k,l})^{-2} \Big( |k_1|\; |\lambda_n|^2|\lambda_l|^2 + |l_1|\; |\lambda_n|^2|\lambda_k|^2 +| n_1|\; |\lambda_k|^2|\lambda_l|^2\Big)\; .
$$
Then, we use the bound on $|\Delta_n^{k,l}|$ to get : 
$$
|F_{n,n}(t)| \leq C\|g_n\|_{L^\infty}^4 |n_1| \sum_{k+l = n} (|k_1l_1n_1|)^{-2} \Big( |k_1|\; |\lambda_n|^2|\lambda_l|^2 + |l_1|\; |\lambda_n|^2|\lambda_k|^2 +| n_1|\; |\lambda_k|^2|\lambda_l|^2\Big)\; .
$$
Separating the sum into two different components, using the symmetry on $k$ and $l$, we have : 
$$
|F_{n,n}(t)| \leq C\|g_n\|_{L^\infty}^4 \left( |n_1|^{-1}|\lambda_n|^2 \sum_{l} |l_1|^{-2}  |\lambda_l|^2 + \sum_{k+l = n}(|k_1l_1|)^{-2}|\lambda_k|^2|\lambda_l|^2\right) \; .
$$
We then sum this quantity over $n$ having previously multiplied it by $|n_1|\;|n|^{2s}$. Since $(|k_1l_1|)^{-2} \leq C|n_1|^{-2}$ and $|n|^{2s} \leq C( |k|^{2s}+|l|^{2s})$ when $k+l=n$ : 
$$
\sum_n |n_1|\;|n|^{2s} |F_{n,n}(t)| \leq C\|g_n\|_{L^\infty}^4 \left( \sum_n |n|^{2s}|\lambda_n|^2 \right)^2 \; .
$$
We then remark that
$$
\frac{\|g_n\|_{L^\infty}^2}{E(|g_n|^2)}E(|g_n|^2) \sum_n |n|^{2s}|\lambda_n|^2  = \frac{\|g_n\|_{L^\infty}^2}{E(|g_n|^2)}\|u_0 \|_{L^2_\Omega ,H^s}^2 \leq \frac{\|g_n\|_{L^\infty}^2}{E(|g_n|^2)}\|u_0\|_{L^\infty,H^s}< \infty
$$
which concludes the proof. \end{proof}

\begin{proposition} The first order in $ \varepsilon$ of the expansion of $E(u_nu_mu_p)$ is equal to 
$$
F_{n,m,p}(t) = e^{i(\omega_n+\omega_m + \omega_p)t}\delta_{n+m+p}^0 \int_{0}^t H_{n,m,p}(t')dt'
$$
and the quantity
$$
\sum_{n,m,p} \sqrt{|n_1m_1p_1|}|nmp|^{s}|F_{n,m,p}(t)| 
$$
is bounded uniformly in time.
\end{proposition}

\begin{proof} The first order in $\varepsilon$ of the quantity $\partial_t E(e^{-i(\omega_n + \omega_m + \omega_p)t}u_nu_mu_p)$ is equal to 
$$
\partial_t e^{-i(\omega_n + \omega_m + \omega_p)t} F_{n,m,p}
$$
in the one hand and to
$$
\delta_{n+m+p}^0  H_{n,m,p}(t)
$$
on the other hand. Hence, as $F_{n,m,p}(t=0) = 0$, we get :
$$
F_{n,m,p}(t) = e^{i(\omega_n+\omega_m + \omega_p)t}\delta_{n+m+p}^0 \int_{0}^t H_{n,m,p}(t')dt' \; .
$$
The integral over time of $H$ is given by : 
\begin{eqnarray*}
\int_{0}^t H_{n,m,p}(t')dt' & = & -\frac{e^{-i(\omega_n + \omega_m + \omega_p)t}-1}{\omega_n + \omega_m + \omega_p} \left( E(|g_n|^2) \Big( n_1 |\lambda_m|^2|\lambda_p|^2 + m_1 |\lambda_p|^2|\lambda_n|^2 + p_1 |\lambda_n|^2|\lambda_m|^2 \Big)+ \right. \\
 & & \left. (E(|g_n|^4 - 2 E(|g_n|^2)^2) \Big( \delta_m^p m_1 |\lambda_m|^4 + \delta_p^n p_1 |\lambda_p|^4 + \delta_n^m n_1 |\lambda_n|^4 \Big) \right)\; .
\end{eqnarray*}
As $n+m+p=0$, we have that $\omega_n + \omega_m + \omega_p = \Delta_{-p}^{n,m}$, therefore we get the bound : 
$$
|\omega_n + \omega_m + \omega_p|\geq 3|n_1m_1p_1|
$$
which gives :
$$
|F_{n,m,p}(t)| \leq C \|g_n\|_{L^\infty}^4 \delta_{n+m+p = 0} \Big( \frac{|\lambda_m|^2|\lambda_p|^2}{|m_1p_1|} +\frac{|\lambda_p|^2|\lambda_n|^2}{|p_1n_1|} +\frac{|\lambda_n|^2|\lambda_m|^2}{|n_1m_1|} \Big)\; .
$$
Multiplying this formula by $\sqrt{|n_1m_1p_1|}|nmp|^{s}$ and summing it over $n$, $m$, and $p$ gives, by symmetry over $n$,$m$ and $p$ : 
$$
\sum_{n,m,p} \sqrt{|n_1m_1p_1|}|nmp|^{s}|F_{n,m,p}(t)| \leq 3  \|g_n\|_{L^\infty}^4 \sum_{n+p+m =0} \sqrt{|n_1m_1p_1|}|nmp|^{s} \frac{|\lambda_m|^2|\lambda_p|^2}{|m_1p_1|} \; .
$$
We then use that $\sqrt{|n_1|}|n|^s \leq C_s ( \sqrt{|m_1|}|m|^s + \sqrt{|p_1|}|p|^s)$ when $n+m+p = 0$ and the symmetry over the sum in $m$ and $p$ to get : 
$$
\sum_{n,m,p} \sqrt{|n_1m_1p_1|}|nmp|^{s}|F_{n,m,p}(t)| \leq C_s  \|g_n\|_{L^\infty}^4 \sum_{m,p} |m|^{2s} |\lambda_m|^2 \frac{|p|^s |\lambda_p|^2}{\sqrt{|p_1|}}
$$
and since the sum $\sum |m|^{2s} |\lambda_m|^2$ is finite and $\frac{|p|^s}{\sqrt{|p_1|}}\leq |p|^{2s}$, we have that
$$
\sum_{n,m,p} \sqrt{|n_1m_1p_1|}|nmp|^{s}|F_{n,m,p}(t)| <\infty \; .
$$ \end{proof}

\subsection{Almost conservation of the moments}\label{sub-almcon}
In this subsection, we make further remarks about the result.

First, we introduce the sequence 
$$
u^N = \sum_n g_n \lambda_n^N e^{inz}
$$
where the $g_n$ are $L^\infty$ independent identically distributed random variables, we assume that 
$$
E(|g_n|^2)= 1\mbox{ and } E(|g_n|^4) = 2 \; ,
$$
and we write
$$
\lambda_n^N = \left \lbrace{ \begin{tabular}{ll}
$ \lambda^N > 0 $ & \mbox{ if }$\max(|n_1|,|n_2|) \leq N$ \\
$0$ & \mbox{ otherwise, }\end{tabular}} \right.
$$
where $\lambda^N$ goes to $0$ when $ N$ goes to $\infty$.

We introduce this sequence because if in the term of second order in $\varepsilon$ in the expansion of $E(|u_n(t)|^2)$, the $\lambda_n$ did not depend on $n$ (and that $E(|g_n|^4) = 2E(|g_n|^2)^2$) then this term would be formally null. However, since the initial datum must be in $L^\infty,H^s$ with $s>1$, the $\lambda_n$ can not be constant with regard to $n$ within our framework. 

We denote by $F_n^N(t)$ the term of second order in $\varepsilon$ in the development of $E(|u_n(t)|^2)$.

\begin{proposition} For all time $t$ and all $n$, $F_n^N(t)$ goes to $0$ when $N$ goes to $\infty$. \end{proposition}

\begin{remark}Even though the sequence $u^N$ converges towards $0$ in some sense, it is not sufficient to conclude. Indeed, $u^N$ converges toward $0$ in $L^p_\Omega,H^\sigma$, with $\sigma<-1$ but not when $\sigma$ is higher, for instance, one can take any $p$, $\sigma = -1$ and $\lambda^N = (\log N)^{-1/2}$. Besides, $F_n^N$ is controlled by the $L^2,H^s$ norm of the initial datum when $s>1$, but it does not seem to be controlled by the $L^p,H^\sigma$ with $\sigma <-1$, because even though there is some regularization due to the absence of resonances, it only affects the variation on the first space variable $x$.\end{remark}

\begin{proof} We can write $F_n^N(t)$ as : 
$$
F_n^N(t) = - n_1  \sum_{k +l = n} \frac{\cos(\Delta_n^{k,l}t)-1}{(\Delta_n^{k,l})^2} \Big( k_1 |\lambda_n^N|^2|\lambda_l^N|^2 + l_1 |\lambda_n^N|^2|\lambda_k^N|^2 - n_1 |\lambda_k^N|^2|\lambda_l^N|^2\Big)
$$
thanks to the equality on the moments of $g_n$.

We treat the case when $\max(|n_1|,|n_2|)$ is less than $N$ as $N$ goes to $\infty$ and $n$ is fixed. Then, whenever $l$ and $k$ satisfy the same property, we have $\lambda_l^N = \lambda_k^N = \lambda_n^N$ and hence 
$$
\Big( k_1 |\lambda_n^N|^2|\lambda_l^N|^2 + l_1 |\lambda_n^N|^2|\lambda_k^N|^2 - n_1 |\lambda_k^N|^2|\lambda_l^N|^2\Big) =|\lambda^N|^4 (k_1+l_1 -n_1) = 0\; .
$$
By symmetry over $k$ and $l$ we get the following bound for $F_n^N$ : 
$$
|F_n^N(t)| \leq C \sum_{\max(|l_1|,|l_2|) > N, k+l=n} \frac{|n_1|}{|\Delta_n^{k,l}|^2} |l_1| |\lambda_k^N|^2|\lambda_n^N|^2
$$
where several terms of the sum have disappeared since when $\max(|l_1|,|l_2|) > N$, $\lambda_l^N = 0$. If $\max(|k_1|,|k_2|) > N$ then $\lambda_k^N = 0$ so we can erase another round of terms in the sum. We get
$$
|F_n^N(t)| \lcurl \sum_A \frac{|n_1|}{|\Delta_n^{k,l}|^2} |l_1| |\lambda_k^N|^2|\lambda_n^N|^2
$$
where $A$ is the set of $(k,l)$ such that $\max(|k_1|,|k_2|) \leq N$, $\max (|l_1|,|l_2|) >N$ and $k+l= n$.

Then, we divide the sum between the cases $|l_1| > N $ and $|l_2|> N$ and use the bound
$$
|\Delta_n^{k,l}|^2 \geq 9 |n_1 k_1 l_1|^2 \; .
$$
We get : 
$$
|F_n^N(t)| \leq C \frac{(\lambda^N)^4}{|n_1|}\left(  \sum_{A_2}\frac{1}{|l_1|k_1^2} + \sum_{A_1}\frac{1}{|l_1|k_1^2}\right) 
$$
where 
$$
A_2 = \lbrace (k,l) \; |\; |l_2|> N, |k_1|\leq N,|k_2|\leq N, l+k=n \rbrace
$$
and
$$
A_1 = \lbrace (k,l) \; |\; |l_1|> N, |k_1|\leq N,|k_2|\leq N, l+k=n \rbrace \; .
$$
For the first sum, we have only at most $|n_2|$ choices for $k_2$. Indeed, if $n_2 \geq 0$ then we have $-N \leq k_2 \leq N$ and as $k_2 = n_2-l_2$ and $|l_2|> N$, then $k_2 > n_2+ N \geq N$ or $k_2 < n_2-N$, which can be combined as $ -N \leq k_2 < n_2 -N$ with a similar result in the case $n_2 \leq 0$ ($N+n_2 \leq k_2 \leq N$). For the second sum we use the bound on $|l_1|$ and the fact that we have at most $2N$ choices for $ k_2$. For both, we use that $\sum k_1^{-2}$ is finite. This gives : 
$$
|F_n^N(t)| \leq C \frac{(\lambda^N)^4}{|n_1|}\left( |n_2| \sum_{k_1 }\frac{1}{k_1^2} + \sum_{k_1 }\frac{1}{k_1^2}\right) \leq C_n (\lambda^N)^4 \; ,
$$
which concludes the proof, as if $n$ is fixed, above a certain rank $N \geq |n|$ and $\lambda^N$ goes to $0$.
\end{proof}

\begin{remark}The same result is true if we replace $F_n^N(t)$ by the term of first order in $\varepsilon$ in the development of $E(u_nu_mu_p)$ without assuming that $\lambda^N$ goes to $0$. \end{remark}

\begin{remark}The fact that $F_n(t)$ is formally equal to $0$ when the $\lambda_n$ are constant with regard to $n$ and $E(|g_n|^4) = 2 E(|g_n|^2)$ results from the fact that the measure $\mu$ induced by
$$
\sum_n g_n e^{inz}
$$
where $g_n$ are complex centred Gaussian variables is formally invariant by the flow of KP-II. Indeed, its "finite-dimensional" version should look like
$$
d\mu(u)  = Ce^{-c \|u\|_{L^2}^2} dL(u)
$$
with $L$ the Lebesgue measure. As the $L^2$ norm of the solution is an invariant of KP-II and KP-II is a Hamiltonian equation, this is formally invariant, the main problem in passing to effective invariance is that KP-II is not globally well-posed on the support of $\mu$, which is $\cap_{\sigma < -1}H^\sigma$.

Nevertheless, if we wish to develop the moments of higher order than $2$ or $3$ of the solution in order to get a better idea of the evolution of the law of the solution, it seems that we should get a development of this form : 
$$
\partial_t E\left( \prod_{i=1}^{2p} u_{n_i} \right) = \varepsilon^2 \delta_{\sum_i n_i}^0 I_{2p}((|\lambda_k|^2)_k,(E(|g_n|^{2j}))_{j=1,\hdots,p+1}) + \varepsilon^4 R_{2p}
$$
that is, null terms of order 1 and 3, the term of order 2 null if the sum $\sum_i n_i$ is different from zero and that depends only on the sequence $(|\lambda_n|^2)_n$ and the $p+1$ first even moments of the law of $g$, and besides
$$
\partial_t E\left( \prod_{i=1}^{2p+1} u_{n_i} \right) = \varepsilon \delta_{\sum_i n_i}^0 I_{2p+1}((|\lambda_k|^2)_k,(E(|g_n|^{2j}))_{j=1,\hdots,p+1}) + \varepsilon^3 R_{2p+1}
$$
and as long as we do not take care of the remainders $R$, the $I$ should be null when their arguments coincide with these of the formally invariant measure, that is when $|\lambda_n|$ does not depend on $n$ and the first $2p+2$ moments of $g_n$ are equal to a complex Gaussian first $2p+2$ moments.
\end{remark}

\paragraph{Acknowledgements}

The author would like to thank Nikolay Tzvetkov for valuable discussions on the subject.

\bibliographystyle{amsplain}
\bibliography{bibnormal} 
\nocite{*}

\end{document}